\documentclass[11pt]{amsart}
\usepackage[margin=1in]{geometry}

\usepackage[utf8]{inputenc}									
\usepackage{amsmath}
\usepackage{amssymb, amsfonts}
\usepackage{amsthm}
\usepackage{hyperref}
\usepackage{mathtools}
\usepackage{tikz}
\usepackage{enumitem}
\usepackage{makecell}
\setcellgapes{4pt}

\newtheorem{thm}{Theorem}
\newtheorem{cor}[thm]{Corollary}
\newtheorem{lem}[thm]{Lemma}
\newtheorem{prop}[thm]{Proposition}

\theoremstyle{definition}
\newtheorem{defn}[thm]{Definition}
\newtheorem{ex}[thm]{Example}
\newtheorem{rem}[thm]{Remark}
\newtheorem{nota}[thm]{Notation}

\usepackage{soul}
\usepackage{tikz-cd}

\title{The Bernardi formula for non-transitive deformations of the braid arrangement}
\author{Ankit Bisain}
\address{MIT PRIMES, Department of Mathematics, Massachusetts Institute of Technology, Cambridge MA 02142, USA}
\email{ankit.bisain@gmail.com}
\author{Eric J. Hanson}\thanks{This is the final manuscript of an open-source paper published in {\it The Electronic Journal of Combinatorics} {\bf 28} (2021), no. 4. DOI:\href{https://doi.org/10.37236/10233}{10.37236/10233}.}
\address{Department of Mathematics, Brandeis University, Waltham MA 02453, USA}\email{ehanson4@brandeis.edu}

\date{October 8, 2021}

\newcommand{\cadet}{\mathsf{cadet}}
\newcommand{\lsib}{\mathsf{lsib}}
\newcommand{\rsib}{\mathsf{rsib}}

\newcommand{\parent}{\mathsf{parent}}
\newcommand{\child}{\mathsf{child}}

\renewcommand{\S}{\mathbf{S}}
\newcommand{\T}{\mathcal{T}}

\begin{document}

\maketitle

\begin{abstract}
    Bernardi has given a general formula for the number of regions of a deformation of the braid arrangement as a signed sum over \emph{boxed trees}. We prove that each set of boxed trees which share an underlying (rooted labeled plane) tree contributes 0 or $\pm 1$ to this sum, and we give an algorithm for computing this value. For Ish-type arrangements, we further construct a sign-reversing involution which reduces Bernardi's signed sum to the enumeration of a set of (rooted labeled plane) trees. We conclude by explicitly enumerating the trees corresponding to the regions of Ish-type arrangements which are nested, recovering their known counting formula.
\end{abstract}

\section{Introduction}\label{sec:intro}

Hyperplane arrangements, or collections of codimension-1 subspaces of a vector space (typically $\mathbb{R}^n$) appear in many different mathematical contexts. For example, the elements of any finite Weyl group $W$ can be interpreted as (compositions of) reflections across certain hyperplanes in $\mathbb{R}^n$. The resulting ``Coxeter arrangement'' then provides a geometric interpretation of the so-called ``weak order'' on $W$. This hyperplane arrangement also reappears in the representation theory of finite-dimensional algebras as the ``stability diagram'' of a ``preprojective algebra''. See \cite{thomas}.

A common problem in enumerative combinatorics is to count the connected components (called \emph{regions}) of the complement of a hyperplane arrangement.  In 1975, Zaslavsky showed that this can be done by plugging -1 into the characteristic polynomial of the arrangement \cite{zaslavsky}. This has helped motivate a large body of work aimed at computing the chacteristic polynomials of hyperplane arrangements. See e.g. \cite{stanley_supersolvable,terao,athanasiadis}. Another method for counting the regions of a hyperplane arrangement is to establish a bijection with some well-known set of combinatorial objects. This approach has the advantage of offering a somewhat direct explanation as to why two sets of objects are enumerated by the same formula.
    
In a main result of \cite{bernardi}, Bernardi gives a general formula for the number of regions of a certain type of hyperplane arrangement, called a \emph{deformation of the braid arrangement}. This formula is a signed sum over so called \emph{boxed trees}, which are partitions of the nodes of certain rooted labeled plane trees. (See Definition \ref{def:boxed} below for the precise definition.) Under a condition known as \emph{transitivity} (Definition \ref{def:transitive}), Bernardi further shows that the signed sum reduces to an enumeration, thereby recapturing and generalizing several known counting formulas. The aim of this paper is to reduce Bernardi's formula for arrangements which are not transitive. Our main result is an algorithm for computing the restriction of Bernardi's sum to the set of boxed trees sharing an underlying (rooted labeled plane) tree (see Theorem \ref{thm:characterization}). We then reduce the signed sum to an enumeration for the class of \textit{Ish-type} arrangements (Definition \ref{def:Ish}). This class notably contains the \emph{Ish arrangement}, which is known to have its regions enumerated by the Cayley numbers (see \cite{AR} and the discussion in Section \ref{sec:background}).
    
    \subsection{Organization and Main Results}\label{sec:organization}
    This paper is organized as follows: In Section \ref{sec:background}, we review the relevant background and definitions. In Section \ref{sec:treeContributions}, we study the contribution of all \emph{$\S$-boxings} (Definition \ref{def:boxed}) of a given (rooted labeled plane) tree to the signed sum of Bernardi and prove our first main theorem:
    \begin{thm}[Theorem \ref{thm:characterization}, simplified version]\label{thmA}
       For an arbitrary deformation of the braid arrangement, the contribution of the set of boxed trees sharing an underlying (rooted labeled plane) tree to the formula of Bernardi is 0 or $\pm 1$. Moreover, there exists an algorithm for computing this contribution.
    \end{thm}
    We believe the algorithm in Theorem~\ref{thmA} can be implemented in polynomial time. See Remark~\ref{rem:complexity} for further discussion.
    We also note that trees with a contribution of $-1$ do not appear in the transitive case, but we suspect that they are ubiquitous in the non-transitive case. See Remark \ref{rem:transitiveOnly}.
    
   In Section \ref{sec:Ish}, we define \emph{almost transitive} arrangements (Definition \ref{def:Almost Transitive}) and specialize the result of Theorem \ref{thm:characterization} to these arrangements (Proposition \ref{lem:IshContribution}). In Section \ref{sec:IshParams}, we further restrict to Ish-type arrangements (Definition \ref{def:Ish}). In this case, we consider four parameters for every tree contributing $\pm 1$ to the signed sum: the \emph{lower 1-length}, \emph{lower inefficiency}, \emph{upper 1-length}, and \emph{upper inefficiency} (Definitions \ref{not:ell} and \ref{def:inefficient}). We then construct a series sign-reversing involutions allowing us to prove our second main theorem:
   
    \begin{thm}[Theorem \ref{thm:involution}]\label{thmB}
    The regions of an Ish-type arrangement are equinumerous with the corresponding rooted labeled plane trees with lower 1-length, lower inefficiency, upper 1-length, and upper inefficiency all 0.
	\end{thm}
	
	In Section \ref{sec:enumeration}, we prove our final main theorem, showing explicitly for any \emph{nested Ish arrangement} (Definition \ref{def:Ish}) that our sign-reversing involutions reduce the Bernardi formula to the known counting formula of \cite{AST}.
	
	\begin{thm}[Theorem \ref{thm:Cayley}]\label{thmC}
	Let $\mathcal{A}_\mathbf{S}$ be a nested Ish arrangement in $\mathbb{R}^n$. For $2 \leq j \leq n$, let $S_{1,k}$ be the set of hyperplanes in $\mathcal{A}_\mathbf{S}$ of the form $x_1 - x_k = s$ for some $s \in \mathbb{R}$. Then the number of regions of $\mathcal{A}_\mathbf{S}$ is given by
	\begin{equation}\label{countIsh}
	r_{\mathbf{S}} = \prod_{k = 2}^n (n + 1 + |S_{1,k}| - k).
	\end{equation}
	\end{thm}
	
\section{Background}\label{sec:background}
In this section, we recall notation, constructions, and background results that will be used in this paper. For detailed definitions and background pertaining to hyperplane arrangements, we refer readers to \cite{stanley}. For the purposes of introduction, we follow much of the exposition in \cite{bernardi}.

We consider \textit{hyperplane arrangements} consisting of hyperplanes in $\mathbb{R}^n$ of the form $$H_{i,j,s}:\text{ }x_i-x_j=s$$ for some $1 \leq i < j \leq n$ and $s \in \mathbb{Z}$. These are known as \textit{deformations of the braid arrangement}, where the \emph{braid arrangement} consists of $\{H_{i,j,0}\}$ for all $1 \leq i < j \leq n$. A common question about a hyperplane arrangement is the number of \textit{regions} into which it divides $\mathbb{R}^n$, where a region is a connected component of the complement of the hyperplanes.

Deformations of the braid arrangement include several families of hyperplane arrangements with historically known counting formulas for their number of regions. Examples include the braid arrangement, \emph{Shi arrangement} \cite{shi}, and \emph{Linial arrangement} \cite{PS}. We refer to \cite[Sections 1-2]{bernardi} for additional examples and references.

Consider a deformation of the braid arrangement $\mathcal{A} = \{x_i - x_j = s\}$ in $\mathbb{R}^n$. We identify $\mathcal{A}$ with the tuple of sets $\mathbf{S} = (S_{i,j})_{1 \leq i < j \leq n}$, where for $1 \leq i < j \leq n$,
    $$S_{i,j} := \{s:(x_i - x_j = s)\in \mathcal{A}\}.$$
This hyperplane arrangement is then called the \emph{$\mathbf{S}$-braid arrangement}, and we write $\mathcal{A} = \mathcal{A}_\mathbf{S}$. The number of regions of $\mathcal{A}_\mathbf{S}$ is denoted $r_\mathbf{S}$. We likewise denote
\begin{equation}\label{eqn:Sminus}
    S_{j,i} := S_{i,j}, \qquad S_{i,j}^- := \{s\geq 0|-s \in S_{i,j}\},\qquad S_{j,i}^- := \{0\}\cup\{s>0|s\in S_{i,j}\}
\end{equation}
For use later, we fix the notation
\begin{equation}\label{eqn:m}
    m := \max\left\{|s|: s \in \bigcup\limits_{1 \leq i<j \leq n} S_{i,j}\right\}.
\end{equation}

In this paper, we consider \emph{Ish-type arrangements} as our prototypical example.

\begin{defn}\label{def:Ish}
Let $\mathbf{S}$ be a tuple of sets as above such that $0 \in S_{i,j}$ for all $1 \leq i < j \leq n$ and $S_{i,j} = \{0\}$ whenever $i\neq 1$. Then $\mathcal{A}_\mathbf{S}$ is called an \emph{Ish-type arrangement}. If in addition $S_{1,j} \subseteq S_{1,k}$ whenever $1 < j < k \leq n$, then $\mathcal{A}_\mathbf{S}$ is called a \emph{nested Ish arrangement}. If $S_{1,j} = \{0,1,\ldots,j-1\}$ for all $1 < j \leq n$, then $\mathcal{A}_\mathcal{S}$ is simply called the ($n$-dimensional) \emph{Ish arrangement}. 
\end{defn}

    \begin{figure}
    \centering
    \begin{tikzpicture}\clip (-1.5,-1) rectangle (1,1);
    
        \draw [dashed,thick](0,-2) -- (0,2);
        \draw [thick](-0.5,-2) -- (-0.5,2);
        \draw [thick](-1,-2) -- (-1,2);
        
        \begin{scope}[rotate={-60}]
        \draw [dashed,thick](0,-2)--(0,2);
        \draw [thick](-0.5,-2)-- (-0.5, 2);
        \end{scope}
        
        \begin{scope}[rotate={60}]
        \draw [dashed,thick](0,-2)--(0,2);
        \end{scope}
    \end{tikzpicture}
    \caption{The projection of the Ish arrangement for $n=3$ onto the plane $x_0+x_1+x_2 = 0$, viewed from the direction $(1,1,1)$. The three hyperplanes $x_i-x_j=0$ are drawn as dashed. The complement of the arrangement consists of $16 = 4^2$ connected components.}\label{fig:Ish}
    \end{figure}
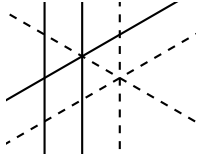

 The Ish arrangement is defined in \cite{armstrong}, and an example with $n=3$ is shown in Figure \ref{fig:Ish}. It is shown in \cite{armstrong} that, like those of the Shi arrangement, the regions of the Ish arrangement are enumerated by the Cayley numbers $(n+1)^{n-1}$. The connection between the regions of the Ish arrangement and those of the Shi arrangement is further explored in  \cite{AR,LRW}, with \cite{LRW} describing an explicit bijection between them. The more general Ish-type and nested Ish arrangements are introduced in \cite{AST}\footnote{In \cite{AST}, an Ish-type arrangement $\mathcal{A}_\mathbf{S}$ is referred to as the ``$\mathcal{N}$-Ish arrangement'' with $\mathcal{N}:= (S_{1,2},S_{1,3},\ldots,S_{1,n})$. Moreover, the definition of a nested Ish arrangement includes the case where there exists a permutation $\sigma$ on $\{2,\ldots,n\}$ so that $S_{1,\sigma(j)}\subseteq S_{1,\sigma(k)}$ whenever $1<j<k\leq n$. Since permuting the axes does not affect the number of regions in a hyperplane arrangement, we have assumed without loss of generality that our nests are always indexed using the identity permutation.}. The authors of that paper prove that nested Ish arrangements are free, and use this fact to show their regions are enumerated by the formula in Equation \ref{countIsh}.

\subsection{The Bernardi Formula}\label{sec:bernardi}

We now give an overview of the constructions and definitions in \cite{bernardi} that lay the foundation of this paper. 

A \emph{rooted tree} is a tree (a connected graph with no cycles) with some vertex designated as the \emph{root}. If $u$ and $v$ are vertices of a rooted tree and the unique path from $v$ to the root goes through $u$, we say $v$ is \emph{above} $u$ and $u$ is \emph{below} $v$. If in addition $u$ and $v$ are connected by an edge, we say $u$ is the \emph{parent} of $v$ and $v$ is one of the \emph{children} of $u$.

An arbitrary vertex in a rooted tree is a called a \textit{node} if it has at least one child and is called a \textit{leaf} otherwise. Under these notions, a \emph{rooted labeled tree} is a rooted tree with the nodes labeled with distinct positive integers from $1$ to the number of nodes.

In this paper, we consider \textit{rooted labeled plane trees}. These are rooted labeled trees with a (left-to-right) ordering imposed on the children of each node. We draw rooted labeled plane trees with the root of the tree at the bottom, each child of a vertex above the original vertex, and the children of any vertex ordered from left to right. Given a node $u$, its \emph{cadet} is its rightmost child that is also a node (if one exists), and is denoted $\cadet(u)$. If $v$ has parent $u$, we define the \emph{left siblings} of $v$ to be the children of $u$ (nodes and leaves) which are to the left of $v$. We denote by $\lsib(v)$ (or $\lsib_T(v)$ if we wish to emphasize the tree $T$) the number of left siblings of $v$. We define \emph{right siblings} and $\rsib(v)$ similarly. An example of these concepts is shown in Figure \ref{fig:tree}.

\begin{figure}
\centering
	\begin{tikzpicture}[scale=0.8]
	    \draw (0,-1.25) -- (-1.5,0);
		\draw (0,-1.25)--(0,0);
		\draw (0,0)--(-2.25,1.5);
		\draw (0,0)--(-0.75,1.5);
		\draw (0,0)--(0.75,1.5);
		\draw (0,0)--(2.25,1.5);
		\draw (-0.75,1.5)--(-0.75,2.75);
		\node[draw,circle,fill=white] at (0,-1.25) {$4$};
		\node[draw,circle,fill=white] at (0,0) {$5$};
		\node[draw,circle,fill=white] at (-0.75,1.5) {$2$};
		\node[draw,circle,fill=white] at (2.25,1.5) {$1$};
		\node[draw,circle,fill=white] at (-0.75,2.75) {$3$};
		\node[draw,circle,fill=white] at (-2.25,1.5) {$6$};
	\end{tikzpicture}
\caption{
A rooted labeled plane tree with 6 nodes. In this tree, the node $4$ is the root, $\cadet(4)=5$, $\cadet(5) = 1$, $\lsib(5)=1$, $\lsib(1)=3$, $\lsib(3)=0$, $\text{parent}(6)=5$, and the node $1$ is a right sibling of the node $6$.}\label{fig:tree}
\end{figure}
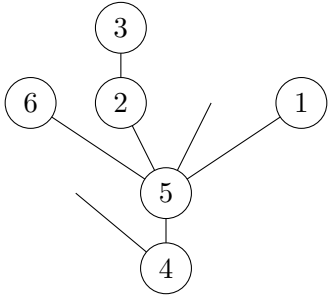

\begin{nota}\label{notation:tmn}
    Let $n$ and $m$ be positive integers. We denote by $\mathcal{T}^{(m)}(n)$ the set of rooted labeled plane trees with $n$ nodes such that each node has $m+1$ children.
\end{nota}

\begin{rem}\label{rem:omitrightsibs}
    We can consider the tree in Figure \ref{fig:tree} as an element of $\mathcal{T}^{(3)}(6)$ by ``adding right siblings". That is, the nodes 1, 3, and 6 are considered to each have 4 children (all leaves), the node 2 is considered to have 3 children (all leaves) which are to the right of 3, and the node 4 is considered to have 2 children (both leaves) to the right of 5. For any node $v$, we then observe that both $\cadet(v)$ (if it exists) and $\lsib(v)$ are the same in the tree drawn in Figure \ref{fig:tree} and the ``larger'' tree in $\mathcal{T}^{(3)}(6)$. We adopt this convention of ``omitting right siblings'' in many of our examples for improved readability.
\end{rem}

For the remainder of this section, let us fix an arbitrary deformation of the braid arrangement $\mathcal{A}_\mathbf{S}$. We then take $n$ to be the dimension of $\mathcal{A}_{\mathbf{S}}$ and $m$ as in Equation \ref{eqn:m}.

\begin{defn}\cite[Definition 4.1]{bernardi}\label{def:S-cadet}
    A sequence $(v_1, v_2, \ldots, v_k)$ of nodes in a tree $T \in \mathcal{T}^{(m)}(n)$ is a \textit{cadet sequence} if $v_j=\cadet(v_{j-1})$ for all $1 < j \leq k$. If in addition $\sum_{p=i+1}^{j} \lsib(v_p) \not\in S^{-}_{v_i,v_j}$ for all $1 \leq i < j \leq k$, then the sequence $(v_1, v_2, \ldots, v_k)$ is called an \emph{$\mathbf{S}$-cadet sequence}.
\end{defn}

\begin{ex}
     The tree in Figure~\ref{fig:tree} has cadet sequences $(4,5,1)$, $(2,3)$, and $(6)$. Subsequences of these which are convex with respect to the relation ``is the cadet of'', for example $(4,5)$ and $(3)$, are also cadet sequences. Now suppose that $\mathcal{A}_\mathbf{S}$ is a deformation of the braid arrangement in $\mathbb{R}^6$ so that $S_{4,5}^- = \{0\}$, $S_{5,1}^- = \{0,2\}$, and $S_{4,1}^- = \{0,2,4\}$. Then both $(4,5)$ and $(5,1)$ are $\mathbf{S}$-cadet sequences, while $(4,5,1)$ is not.
\end{ex}

\begin{rem}
    We note that a cadet sequence is uniquely determined by the nodes it contains. Thus, depending on context, we will freely move between notating cadet sequences by $(v_1,\ldots,v_k)$ and by $\{v_1,\ldots,v_k\}$.
\end{rem}

\begin{defn}\cite[Definition 4.1]{bernardi}\label{def:boxed}
    A \emph{boxed tree} is a pair $(T,B)$, where $T \in \mathcal{T}^{(m)}(n)$ and $B$ is a partition of the nodes of $T$ into cadet sequences. We say that $(T,B)$ is \emph{$\mathbf{S}$-boxed} if each cadet sequence of $B$ is also an $\mathbf{S}$-cadet sequence. The set of $\mathbf{S}$-boxed trees is denoted~$\mathcal{U}_{\mathbf{S}}$.
\end{defn}

Given a boxed tree (resp. $\S$-boxed tree) $(T,B)$, we will sometimes refer to $B$ as a \emph{boxing} (resp. \emph{$\mathbf{S}$-boxing}) of $T$ and refer to the partition elements of $B$ as \emph{boxes} (resp. \emph{$\S$-boxes}). Examples of $\mathbf{S}$-boxings are given in Example \ref{ex:contributions} below.

\begin{thm}\cite[Theorem 4.2]{bernardi}\label{thm:signedFormula}
The number of regions $r_\S$ of the arrangement $\mathcal{A}_\mathbf{S}$ is given by \begin{equation}\label{eqn:Bernardi}r_{\mathbf{S}}=\sum\limits_{(T,B) \in \mathcal{U}_\mathbf{S} } (-1)^{n-|B|},\end{equation}
    where $|B|$ is the number of partition elements in $B$.
\end{thm}

We refer to Equation \ref{eqn:Bernardi} as the \emph{Bernardi formula}. While the formula holds in general, there are many hyperplane arrangements with known explicit counting formulas (for example, the Ish arrangement). One of the main results in \cite{bernardi} is to recover such formulas when the set $\mathbf{S}$ satisfies a condition called \emph{transitivity}. This beautifully unifies and expands upon many known results, in particular answering a question of Gessel. (See \cite[Section 2.3]{bernardi} and \cite[Section 1]{gessel} for further discussion.)

\begin{defn}\cite[Definition 4.3]{bernardi}\label{def:transitive}
    We call the tuple $\mathbf{S}$ \textit{transitive} if for all distinct $i,j,k \in [n]$ and for all nonnegative integers $s\not\in S_{i,j}^{-}$ and $t \not\in S_{j,k}^{-}$, we have $s+t \not\in S_{i,k}^{-}$.
\end{defn}

\begin{thm}\cite[Theorem 4.6]{bernardi}\label{thm:transitiveFormula}
   If $\mathbf{S}$ is transitive, then there exists a set of trees $\mathcal{T}_{\mathbf{S}}\subseteq\mathcal{T}^{(m)}(n)$ such that $r_\mathbf{S} = \sum_{(T,B) \in    \mathcal{U}_\mathbf{S} } (-1)^{n-|B|}=|\mathcal{T}_{\mathbf S}|$. Moreover, let $B_0$ be the partition of $[n]$ into singletons, and consider $\mathcal{T}_\mathbf{S}$ as a subset of $\mathcal{U}_\mathbf{S}$ by identifying each $T \in \mathcal{T}_\mathbf{S}$ with $(T,B_0) \in \mathcal{U}_\mathbf{S}$. Then there is a sign-reversing involution on $\mathcal{U}_{\mathbf{S}}\setminus \mathcal{T}_{\mathbf{S}}$.
\end{thm}

As stated in the theorem, Bernardi's proof involves the construction of a sign-reversing involution; that is, a bijection $\Phi:\mathcal{U}_{\mathbf{S}}\setminus \mathcal{T}_{\mathbf{S}}\rightarrow\mathcal{U}_{\mathbf{S}}\setminus \mathcal{T}_{\mathbf{S}}$ for which $\Phi^2 = \mathsf{Id}$ and if $\Phi(T_1,B_1) = (T_2,B_2)$, then $(-1)^{n-|B_1|} + (-1)^{n-|B_2|} = 0$.
This makes the corresponding terms in the Bernardi formula ``cancel out''. We adopt a similar approach in Section \ref{sec:IshParams}.

\begin{rem}\label{rem:IshNotTrans}
    Let $\mathcal{A}_\mathbf{S}$ be the Ish arrangement with $n\geq 4$. Then $S^{-}_{4,2}=\{0\}$, $S^{-}_{2,1}=\{0,1\}$, and $S^{-}_{4,1}=\{0,1,2,3\}$. Observe that $1 \not\in S^-_{4,2}$, $2 \not\in S^{-}_{2,1}$, but $1+2 \in S^{-}_{4,1}$. Therefore, for $n \geq 4$, the Ish arrangement is not transitive.
\end{rem}

Motivated by Theorem \ref{thm:transitiveFormula} and this remark, we wish to reduce Bernardi's formula to an enumeration when the assumption of transitivity is removed. We do this for Ish-type arrangements in Theorem \ref{thm:involution}, allowing us to recover the known explicit counting formula for nested Ish arrangements in Theorem \ref{thm:Cayley}.

	\section{The Contribution of a Tree}\label{sec:treeContributions}
	
	In this section, we define and characterize the \emph{contribution} of a (rooted labeled plane) tree to the Bernardi formula. We fix for the remainder of this section an arbitrary deformation of the braid arrangement $\mathcal{A}_\mathbf{S}$. As before, we let $n$ be the dimension of $\mathcal{A}_{\mathbf{S}}$ and define $m$ as in Equation \ref{eqn:m}.

	\begin{defn}\label{def:lastFirst}
	Let $T \in \T^{(m)}(n)$. For any sequence $(v_1, v_2, \ldots, v_k)$ of the nodes of $T$, define the \textit{last node} of the sequence to be $v_k$, and the \textit{first node} of the sequence to be $v_1$. We refer to $k$ as the \emph{length} of the sequence.
	\end{defn}
	\begin{defn}\label{def:maxSequence}
	Let $T \in \T^{(m)}(n)$.
	Define a \textit{maximal cadet sequence} of $T$ to be a cadet sequence $(v_1, v_2, \ldots, v_k)$ of $T$ such that all the children of $v_k$ are leaves, and there is no node $u$ for which $\cadet(u)=v_1$.
	\end{defn}
	
	An example of these concepts is shown in Figure \ref{ex:boxedTree}.

	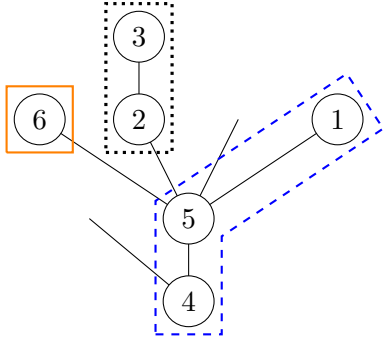
\begin{figure}
	\centering
	\begin{tikzpicture}[scale=0.88]
	    \draw (0,-1.25) -- (-1.5,0);
		\draw (0,-1.25)--(0,0);
		\draw (0,0)--(-2.25,1.5);
		\draw (0,0)--(-0.75,1.5);
		\draw (0,0)--(0.75,1.5);
		\draw (0,0)--(2.25,1.5);
		\draw (-0.75,1.5)--(-0.75,2.75);
		\node[draw,circle,fill=white] at (0,-1.25) {$4$};
		\node[draw,circle,fill=white] at (0,0) {$5$};
		\node[draw,circle,fill=white] at (-0.75,1.5) {$2$};
		\node[draw,circle,fill=white] at (2.25,1.5) {$1$};
		\node[draw,circle,fill=white] at (-0.75,2.75) {$3$};
		\node[draw,circle,fill=white] at (-2.25,1.5) {$6$};
		\draw (-0.5,-1.75)--(0.5,-1.75)--(0.5,-0.2676)--(2.94,1.3613)--(2.388,2.19)--(-0.5,0.2676)--(-0.5,-1.75) [thick,color=blue,dashed];
		\draw (-2.75,1)--(-1.75,1)--(-1.75,2)--(-2.75,2)--(-2.75,1) [thick,color=orange];
		\draw (-1.25,1)--(-0.25,1)--(-0.25,3.25)--(-1.25,3.25)--(-1.25,1) [very thick,dotted,color=black];
	\end{tikzpicture}
	\caption{A rooted labeled plane tree partitioned into maximal cadet sequences. A box is drawn around the nodes of each sequence (but this is not necessarily an $\mathbf{S}$-boxing). The last node of the maximal cadet sequence in the blue/dashed box is $1$, and the first node is $4$.}\label{ex:boxedTree}
	\end{figure}
	
	\begin{rem}\label{rem:comleteCadet}
	    As in Figure \ref{ex:boxedTree}, the nodes of any rooted labeled plane tree can be uniquely partitioned into maximal cadet sequences. Indeed, since no two nodes can have the same cadet and each node has at most one cadet, we can complete any nonempty cadet sequence (in the forward and reverse directions) to a unique maximal cadet sequence. Possibly, such a sequence will contain only one node.
	\end{rem}
	
	Given a (possibly maximal) cadet sequence of a rooted labeled plane tree, we can consider an $\mathbf{S}$-boxing \emph{of the sequence} as in Definition \ref{def:boxed}. More precisely, we have the following.
	
	\begin{defn}
	    Let $T \in \mathcal{T}^{(m)}(n)$ and let $X$ be an arbitrary cadet sequence of $T$. Define an \emph{$\mathbf{S}$-boxing} of $X$ to be a partition $B$ of the nodes in $X$ so that each $Y \in B$ is an $\mathbf{S}$-cadet sequence (of~$T$).
	\end{defn}
	
	We are now prepared to define the \emph{contributions} of a tree and a cadet sequence.

	\begin{defn}\label{def:contribution}
	{ Let $T \in \T^{(m)}(n)$ and let $X$ be a (not necessarily maximal) cadet sequence of $T$. Let $k$ be the length of $X$. Denote by $\mathcal{U}_\mathbf{S}(T)$ and $\mathcal{U}_{\mathbf{S}}(X)$ the sets of $\mathbf{S}$-boxings of $T$ and $X$, respectively. We then define the \emph{contribution} of $T$ and the \emph{contribution} of $X$ as
	$$r_{\mathbf{S}}(T) := \sum_{B \in \mathcal{U}_\mathbf{S}(T)} (-1)^{n-|B|},\qquad\qquad
	r_{\mathbf{S}}(X) :=\sum_{B \in \mathcal{U}_\mathbf{S}(X)}(-1)^{k-|B|}.$$}
\end{defn}

	Observe that, by definition, the Bernardi formula (Equation \ref{eqn:Bernardi}) says that the number of regions of $\mathcal{A}_\S$ is the sum over $T \in \T^{(m)}(n)$ of the contribution of $T$; that is, 
	$$r_\mathbf{S} = \sum_{T \in \T^{(m)}(n)}r_\mathbf{S}(T).$$
	
	\begin{lem}\label{lem:contributionProduct}
	Let $T \in \mathcal{T}^{(m)}(n)$. Then
	the contribution of $T$ is equal to the product of the contributions of its maximal cadet sequences. 
	\end{lem}
	\begin{proof}
	Let $X_1,\ldots,X_t$ be the maximal cadet sequences of $T$. As any $\mathbf{S}$-cadet sequence is necessarily contained in a maximal cadet sequence, choosing an $\mathbf{S}$-boxing of $T$ is equivalent to choosing an $\mathbf{S}$-boxing of each maximal cadet sequence $X_i$. That is, we can identify $\mathcal{U}_\mathbf{S}(T)$ with the product $\mathcal{V}:=\mathcal{U}_\mathbf{S}(X_1)\times\cdots\times \mathcal{U}_\mathbf{S}(X_t)$. We then have
	\begin{eqnarray*}
	    r_\mathbf{S}(T) &=& \sum_{B \in \mathcal{U}_\mathbf{S}(T)} (-1)^{n-|B|}\\
	    &=& \sum_{(B_1,\ldots,B_t) \in \mathcal{V}} \ \prod_{j = 1}^t (-1)^{|X_j| - |B_j|}\\
	    &=& \prod_{j = 1}^t \ \sum_{B_j \in \mathcal{U}_\mathbf{S}(X_j)}(-1)^{|X_j| - |B_j|}\\
	    &=& \prod_{j = 1}^t r_\mathbf{S}(X_j)
	\end{eqnarray*}
	\end{proof}
	
	\begin{defn}\label{def:maxBox}
	Let $T \in \T^{(m)}(n)$. Let $(v_i,v_{i+1},\ldots,v_{j})$ be an $\mathbf{S}$-cadet sequence of $T$ and let $(v_1,v_2,\ldots,v_k)$ be the maximal cadet sequence containing it. We say $(v_i,v_{i+1},\ldots,v_{j})$ is a \emph{maximal $\mathbf{S}$-cadet sequence} if both (a) either $i = 1$ or $(v_{i-1},v_i,\ldots,v_{j})$ is not an $\mathbf{S}$-cadet sequence and (b) either $j = k$ or $(v_i,\ldots,v_j,v_{j+1})$ is not an $\mathbf{S}$-cadet sequence.
	\end{defn}
	
	\begin{ex}\label{ex:contributions}
    In the tree from Figure \ref{ex:boxedTree}, suppose that $S_{2,3}^{-}=S_{4,5}^{-}=S_{5,1}^{-}=\{0\}$, and $S_{4,1}^{-}=\{0,1,2,3,4,5\}$. Then, writing ($\S$-)boxings as the list of sets of nodes in each $\mathbf{S}$-cadet sequence,
    \begin{enumerate}
        \item The only valid $\mathbf{S}$-boxing of the maximal cadet sequence of the orange/solid box is the partition $\{6\}$, for a contribution of $(-1)^{1-1}=1$.
        \item The only valid $\mathbf{S}$-boxing of the maximal cadet sequence of the black/dotted box is the partition $\{2\}\{3\}$, for a contribution of $(-1)^{2-2}=1$.
        \item The valid $\mathbf{S}$-boxings of the maximal cadet sequence of the blue/dashed box are the partitions $\{4\}\{5\}\{1\}$, $\{4,5\}\{1\}$, and $\{4\}\{5,1\}$ for a contribution of $(-1)^{3-3}+(-1)^{3-2}+(-1)^{3-2}=-1$.
        \item The valid $\mathbf{S}$-boxings of the tree are $\{6\}\{2\}\{3\}\{4\}\{5\}\{1\}$, $\{6\}\{2\}\{3\}\{4,5\}\{1\}$, and \\ $\{6\}\{2\}\{3\}\{4\}\{5,1\}$, for a sum of
        \begin{eqnarray*}
            -1 &=& (-1)^{6-6}+(-1)^{6-5}+(-1)^{6-5} \\
            &=&\left((-1)^{1-1}\right)\left((-1)^{2-2}\right)\left((-1)^{3-3}+(-1)^{3-2}+(-1)^{3-2}\right),
        \end{eqnarray*}
        which can be seen to be the product of the contributions of the maximal cadet sequences.
    \end{enumerate}
    Moreover, the maximal $\mathbf{S}$-cadet sequences of the blue/dashed maximal cadet sequence are $\{4,5\}$ and $\{5,1\}$, since $\{4,5,1\}$ is not an $\mathbf{S}$-cadet sequence.
	\end{ex}
	
	\begin{defn}
	   Let $T \in \T^{(m)}(n)$, and let $X$ be a nonempty cadet sequence of $T$. We say that $X$ is \textit{$\mathbf{S}$-connected} if
	    \begin{enumerate}
	        \item Given a maximal $\mathbf{S}$-cadet sequence $Y$ of $T$, either $X \cap Y = \emptyset$ or $Y \subseteq X$.
	        \item No cadet sequence satisfying (1) is properly contained in $X$.
	    \end{enumerate}
	\end{defn}
	
\begin{lem}\label{rem:connected}
    Let $T \in \T^{(m)}(n)$ and let $X = (v_1,\ldots,v_k)$ be a maximal cadet sequence of $T$. Then for all $1 \leq i \leq k$, there exist unique indices $1 \leq i' \leq i \leq i''\leq k$ so that $(v_{i'},\ldots,v_{i},\ldots,v_{i''})$ is an $\mathbf{S}$-connected cadet sequence. That is, any maximal cadet sequence can be uniquely partitioned into $\mathbf{S}$-connected cadet sequences. 
\end{lem}

\begin{proof}
     Let $v_i \in X = (v_1,\ldots,v_k)$ and let $\mathcal{Y}_i$ be the set of cadet sequences containing $v_i$ and satisfying (1) in the definition of an $\mathbf{S}$-connected cadet sequence. Since $X \in \mathcal{Y}_i$ and $\mathcal{Y}_i$ is closed under intersections, we see that $Y_i = \bigcap_{Y \in \mathcal{Y}_i}Y$ is the unique $\mathbf{S}$-connected cadet sequence containing~$v_i$.
\end{proof}
	 
	\begin{lem}\label{lem:contributionProduct2}
	    The contribution of a maximal cadet sequence is equal to the product of the contributions of its $\mathbf{S}$-connected cadet sequences.
	\end{lem}
	\begin{proof}
	    Recall that any $\mathbf{S}$-cadet sequence must be contained within a maximal $\mathbf{S}$-cadet sequence, and any maximal $\mathbf{S}$-cadet sequence must be contained within some $\mathbf{S}$-connected cadet sequence. Thus choosing an $\mathbf{S}$-boxing of a maximal cadet sequence is equivalent to choosing an $\mathbf{S}$-boxing of each of its $\mathbf{S}$-connected cadet sequences. The result then follows using the same argument as in Lemma \ref{lem:contributionProduct}.
	\end{proof}
	
    \begin{ex}
    Figure \ref{fig:maxBoxes} shows a maximal cadet sequence and its maximal $\mathbf{S}$-cadet sequences. (We have suppressed the numbers of left siblings and the definition of $\S$.) This can be partitioned into the $\mathbf{S}$-connected cadet sequences with nodes labeled $\{1,2,3\}$, $\{4\}$, and $\{5,6\}$. \begin{enumerate}
        \item The valid $\mathbf{S}$-boxings of $\{1,2,3\}$ are $\{1\}\{2\}\{3\}$, $\{1,2\}\{3\}$, and $\{1\}\{2,3\}$, so the contribution is $(-1)^{3-3}+(-1)^{3-2}+(-1)^{3-2}=-1$.
        \item The only valid $\mathbf{S}$-boxing of $\{4\}$ is $\{4\}$, for a contribution of $(-1)^{1-1}=1$.
        \item The valid $\mathbf{S}$-boxings of $\{5,6\}$ are $\{5\}\{6\}$ and $\{5,6\}$, for a contribution of $(-1)^{2-2}+(-1)^{2-1}=0$.
        \item The valid $\S$-boxings of the entire maximal cadet sequence are:
        \begin{eqnarray*}
        \{1,2,3\}\{4\}\{5\}\{6\}, \qquad&\{1,2,3\}\{4\}\{5,6\},\qquad&\{1,2\}\{3\}\{4\}\{5\}\{6\},\\
        \{1,2\}\{3\}\{4\}\{5,6\},\qquad&\{1\}\{2,3\}\{4\}\{5\}\{6\},\qquad&\{1\}\{2,3\}\{4\}\{5,6\},
        \end{eqnarray*} for a contribution of $$(-1)^{6-4}+(-1)^{6-3}+(-1)^{6-5}+(-1)^{6-4}+(-1)^{6-5}+(-1)^{6-4}$$
        $$=((-1)^{3-3}+(-1)^{3-2}+(-1)^{3-2})((-1)^{1-1})((-1)^{2-2}+(-1)^{2-1})=0$$
    \end{enumerate}
    \end{ex}
    
    \begin{figure}
    \centering
    \begin{tikzpicture}
        \begin{scope}[rotate={-90}]
	    \draw (0,1*1.25)--(0,6*1.25);
		\node[draw,circle,fill=white] at (0,1*1.25) {$1$};
		\node[draw,circle,fill=white] at (0,2*1.25) {$2$};
		\node[draw,circle,fill=white] at (0,3*1.25) {$3$};
		\node[draw,circle,fill=white] at (0,4*1.25) {$4$};
		\node[draw,circle,fill=white] at (0,5*1.25) {$5$};
		\node[draw,circle,fill=white] at (0,6*1.25) {$6$};
		\draw (-0.5,0.75)--(0.5,0.75)--(0.5,3)--(-0.5,3)--(-0.5,0.75);
		\draw (-0.6,2)--(0.6,2)--(0.6,4.25)--(-0.6,4.25)--(-0.6,2);
		\draw (-0.5,4.5)--(0.5,4.5)--(0.5,5.5)--(-0.5,5.5)--(-0.5,4.5);
		\draw (0.5,5.75)--(0.5,8)--(-0.5,8)--(-0.5,5.75)--(0.5,5.75);
		\end{scope}
	\end{tikzpicture}
	\caption{A maximal cadet sequence and its maximal $\mathbf{S}$-cadet sequences. Cadet relationships move left to right; that is, if there is an edge from $u$ to $v$ and $v$ is right of $u$, then $\cadet(u) = v$.}\label{fig:maxBoxes}
    \end{figure}
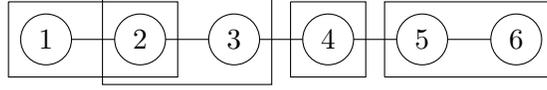

\begin{defn}\label{def:reaches}
	    Let $T \in \T^{(m)}(n)$ and let $(v_1,v_2, \ldots, v_k)$ be an $\mathbf{S}$-connected cadet sequence of $T$ with maximal $\mathbf{S}$-cadet sequences $(X_1,X_2, \ldots, X_{k'})$, in increasing order of index of last node. For convenience, define $X_0 = \{0\} = X_{k'+1}$ and $\parent(v_1) = 0 = \cadet(v_k)$. For $0 \leq i < j \leq k'$, we say $X_i$ \textit{reaches} $X_{j}$ if the parent of the largest indexed node in $X_j\setminus X_{j+1}$ is contained in $X_i$. We say a (not necessarily maximal) $\mathbf{S}$-cadet sequence $X$ \textit{precedes} $X_j$ if the cadet of the last node in $X$ is the last node in $X_j\setminus X_{j+1}$.
\end{defn}

    \begin{ex}\label{ex:reaching}
       Suppose we have a maximal cadet sequence $X = (1,2,3,4,5,6,7)$ with maximal $\mathbf{S}$-cadet sequences $X_1 = \{1,2,3\}$, $X_2 = \{2,3,4,5\}$, $X_3 = \{3,4,5,6\}$, and $X_4 = \{5,6,7\}$. We note that $X$ is actually an $\mathbf{S}$-connected cadet sequence in this case. Now recall that $X_0 = \{0\} = X_5$. Then $X_1$ reaches $X_2$ because $\parent(2) \in X_1$ and $X_1$ reaches $X_3$ because $\parent(4) \in X_1$. Likewise, $X_2$ reaches $X_3$, but $X_2$ does not reach $X_4$. Moreover, the $\mathbf{S}$-box $\{1\}$ precedes $X_2$, any $\mathbf{S}$-box with last node 3 precedes $X_3$, and any $\mathbf{S}$-box with last node $6$ precedes $X_4$.
    \end{ex}

\begin{rem}\label{rem:empty}
We note that the condition $\parent(v_0) = 0$ implies that $X_{0}$ reaches $X_1$ if $X_2\setminus X_1$ contains a single node and $X_{0}$ does not reach any maximal $\mathbf{S}$-cadet sequence otherwise. Likewise, the condition $\cadet(v_k) = 0$ implies that $X_{k'-1}$ reaches $X_{k'}$ if $X_{k'}\setminus X_{k'-1}$ contains a single node and $X_{k'-1}$ does not reach any maximal $\mathbf{S}$-cadet sequence otherwise.
\end{rem}

We are now prepared to state the main result of this section.

\begin{thm}[Theorem \ref{thmA}]\label{thm:characterization}
	 Let $T \in \T^{(m)}(n)$ and let $X = (v_1,v_2, \ldots, v_k)$ be an $\mathbf{S}$-connected cadet sequence of $T$ with maximal $\mathbf{S}$-cadet sequences $(X_1,X_2, \ldots, X_{k'})$, in increasing order of index of last node. Define $X_{0} = \{0\} = X_{k'+1}$. Generate a subsequence $(X_{i_{0}}X_{i_1},\ldots,X_{i_t})$ of $(X_{0},X_1,X_2, \ldots, X_{k'})$ by the following algorithm:
	\begin{enumerate}
	    \item Define $X_{i_0}=X_0$
	    \item For $j \geq 0$ such that $i_{j} \neq k'$, let $X_{i_{j+1}}$ be the first maximal $\mathbf{S}$-cadet sequence reached by $X_{i_{j}}$ but not reached by $X_{i_{\ell}}$ for $\ell<j$, if such a maximal $\mathbf{S}$-cadet sequence exists. If no such maximal $\mathbf{S}$-cadet sequence exists, the algorithm fails.
	    \item If $i_j=k'$, set $t = j$ and terminate the algorithm.
	\end{enumerate}
    Now, if any step of this algorithm fails, the contribution of $X$ is $0$. Otherwise, the contribution of $X$ is given by $r_{\mathbf{S}}(X) = (-1)^{k-t}$.
\end{thm}

Before proving Theorem~\ref{thm:characterization}, we give a detailed example.

\begin{ex}\label{ex:algorithm}
Consider the setup of Example~\ref{ex:reaching}. The algorithm in Theorem~\ref{thm:characterization} first defines $X_{i_0} = X_0$. Then, the algorithm sets $X_{i_1} = X_1$ (see Remark~\ref{rem:empty}) and $X_{i_2} = X_2$. Now, since $X_2$ reaches only $X_3$, which is also reached by $X_1$, the algorithm fails at the $j = 2$ step. This means the contribution of $X$ should be zero. We will show this by considering the possible $\mathbf{S}$-boxings of $X$ in cases.
\begin{enumerate}
    \item We first consider those $\mathbf{S}$-boxings of $X$ which contain $\{1\}$. We then have four subcases:
    \begin{enumerate}
        \item Let $\mathcal{B}$ be the set of $\mathbf{S}$-boxings of $X$ which contain both $\{1\}$ and $\{2\}$. For $B \in \mathcal{B}$ we define a new $\mathbf{S}$-boxing as follows:
        \begin{enumerate}
            \item If there exists $Y \in B$ with $\{3,4\} \subseteq Y$, then let $$B' = (B\setminus \{Y\})\cup \left\{\{3\},Y\setminus\{3\}\right\}.$$
            \item Otherwise, let $Y \in B$ be the $\mathbf{S}$-box with $4 \in Y$. Then let
            $$B' = (B\setminus Y) \cup \{\{3\}\cup Y\}.$$
        \end{enumerate}
        We observe that the association $B \mapsto B'$ is an involution on $\mathcal{B}$. Moreover, for all $B \in \mathcal{B}$, the values of $|B|$ and $|B'|$ will differ by 1. We conclude that the $\mathbf{S}$-boxings in $\mathcal{B}$ together contribute 0 to value of $r_\mathbf{S}(X)$. (Note in particular that $\{2\}$ does not precede any maximal $\mathbf{S}$-cadet sequence. We will use this fact to generalize this argument in the proof of Theorem~\ref{thm:characterization}.)
        \item Let $\mathcal{C}$ be the set of $\mathbf{S}$-boxings of $X$ which contain both $\{1\}$ and $\{2,3\}$. We will consider $\mathcal{C}$ later.
        \item By replacing 3 and 4 with 5 and 6, reasoning analogous to case (1a) shows that the $\mathbf{S}$-boxings of $X$ containing both $\{1\}$ and $\{2,3,4\}$ together contribute 0 to $r_\mathbf{S}(X)$.
        \item By replacing 3 and 4 with 6 and 7, reasoning analogous to case (1a) shows that the $\mathbf{S}$-boxings of $X$ containing both $\{1\}$ and $\{2,3,4,5\}$ together contribute 0 to $r_\mathbf{S}(X)$.
    \end{enumerate}
    \item Reasoning analogous to case (1a) shows that the $\mathbf{S}$-boxings of $X$ containing $\{1,2\}$ and together contribute 0 to $r_\mathbf{S}(X)$. (We note that $\{1,2\}$ does not precede any maximal $\mathbf{S}$-cadet sequence.)
    \item Let $\mathcal{D}$ be the set of $\mathbf{S}$-boxings of $X$ which contain $\{1,2,3\}$. Again, reasoning analogous to case (1a) shows that the $\mathbf{S}$-boxings in $\mathcal{C}\cup\mathcal{D}$ contribute 0 to $r_\mathbf{S}(X)$.
\end{enumerate}
Together, the casework above shows that $r_\mathbf{S}(X) = 0$, as desired.
\end{ex}

We now formalize and extend the arguments from Example~\ref{ex:algorithm} to prove Theorem~\ref{thm:characterization}. The resulting
proof uses ideas similar to that of \cite[Theorem 4.6]{bernardi}.
	
\begin{proof}[Proof of Theorem~\ref{thm:characterization}]
    We first claim that if $X_2\setminus X_1$ contains more than a single node, then the algorithm fails and the contribution of $X$ is 0. Indeed, suppose there exist two nodes $v_i, v_{i+1} \in X_1\setminus X_2$. For an $\mathbf{S}$-boxing $B \in \mathcal{U}_\mathbf{S}(X)$, we define a new $\mathbf{S}$-boxing as follows:
    \begin{enumerate}
        \item If $v_i$ is in the $\mathbf{S}$-cadet sequence $Y \in B$ and $v_{i+1}$ is in the $\mathbf{S}$-cadet sequence $Y'\in B$ with $Y \neq Y'$, then replace $Y$ and $Y'$ with $Y \cup Y'$.
        \item If there exists an $\mathbf{S}$-cadet sequence $Y\in B$ containing both $v_i$ and $v_{i+1}$, replace it with $Y' = \{v_j \in Y|j \leq i\}$ and $Y'' = \{v_j \in Y|j \geq i+1\}$.
    \end{enumerate}
    As this construction changes the number of maximal $\mathbf{S}$-cadet sequences by 1, we have a sign-reversing involution on $\mathcal{U}_{\mathbf{S}}(X)$. In particular, the contribution of $X$ is 0. Moreover, we recall that if $X_2\setminus X_1$ contains more than a single node, then $X_0$ does not reach any maximal $\mathbf{S}$-cadet sequence, so the algorithm fails in this case.

Now let $\mathcal{B}\subseteq \mathcal{U}_\mathbf{S}(X)$ be the set of $\S$-boxings $B$ of $X$ for which every $\mathbf{S}$-box in $B$ except the last one (the one containing $v_k)$ precedes some maximal $\mathbf{S}$-cadet sequence $X_j$ for $j \leq k'$.

We claim there is a sign-reversing involution on the $\S$-boxings which are not in $\mathcal{B}$. Indeed, let $B$ be such an $\S$-boxing and let $Y \in B$ be the first $\mathbf{S}$-box that does not precede any maximal $\mathbf{S}$-cadet sequence. Let $v$ the last node of $Y$. Now since $Y$ does not precede any maximal $\mathbf{S}$-cadet sequence, we have that $\cadet(v)$ and $\cadet(\cadet(v))$ are contained in $X_j\setminus X_{j+1}$ for some $j$. As before, we can then construct a sign-reversing involution by combing or splitting the $\mathbf{S}$-boxes containing $\cadet(v)$ and $\cadet(\cadet(v))$.

We have thus far shown that the contribution of $X$ is 
$r_\mathbf{S}(X) = \sum_{B \in \mathcal{B}}(-1)^{k-|B|}.$ We wish to reduce this to the sum over a single $\S$-boxing (of size $t + 1$).

We claim there is a bijection between $\mathcal{B}$ and the set of subsequences $(X_{i_0},\ldots,X_{i_r})$ of \linebreak $(X_0,X_1,\ldots,X_{k'+1})$ for which $X_{i_0} = X_0, X_{i_r} = X_{k'}$, and $X_{i_j}$ reaches $X_{i_{j+1}}$ for all $j$. Moreover, we will show that the number of $\mathbf{S}$-boxes of any $B\in \mathcal{B}$ is the same as the length (i.e., the value of $r+1$) of the corresponding sequence.

For readability, we use $X_{1+i_j}$ for the maximal $\mathbf{S}$-cadet sequence with index $1+i_j$ and $X_{i_{j+1}}$ for the maximal $\mathbf{S}$-cadet sequence with index $i_{j+1}$ (the maximal $\mathbf{S}$-cadet sequence for which $i$ has index $j+1$).

First consider an $\S$-boxing $B = (Y_1,\ldots,Y_r) \in \mathcal{B}$. For $1 \leq j \leq r-1$, let $X_{i_{j+1}}$ be the maximal $\mathbf{S}$-cadet sequence that $Y_j$ precedes. Let $X_{i_0} = X_0$ and $X_{i_1} = X_1$. Note that $X_{i_r} = X_{k'}$ because $v_k \in X_{k'}\setminus X_{k'-1}$, so $Y_r$ must be contained within $X_{k'}$, so the cadet of the last node in $Y_{r-1}$ must be in $X_{k'}$.

To see that $X_{i_j}$ reaches $X_{i_{j+1}}$, let $v$ be the last node in $Y_{j-1}$. Then since $Y_{j-1}$ precedes $X_{i_j}$, we have $\cadet(v) \in X_{i_j}\setminus X_{1+i_{j}}$. Moreover, we have $\cadet(v) \in Y_{j}$ which means $Y_{j} \subseteq X_{i_j}$. Now let $u$ be the last node in $Y_{j}$. Then since $Y_{j}$ precedes $X_{i_{j+1}}$, we have $\cadet(u) \in X_{i_{j+1}}\setminus X_{1+i_{j+1}}$, and $\cadet(u)$ is also the largest indexed node with this property. Now, since the node $u$ with this property is contained in $Y_{j}$, it is contained in $X_{i_j}$, so $X_{i_j}$ reaches $X_{i_{j+1}}$ by definition.

Now consider a sequence $(X_{i_0},\ldots,X_{i_r})$ so that $X_{i_0} = X_0, X_{i_r} = X_{k'}$, and $X_{i_j}$ reaches $X_{i_{j+1}}$ for all $j$. For $j \in [r-1]$, define $Y_j$ to have last node $u \in X_{i_{j}}$ so that $\cadet(u) \in X_{i_{j+1}}\setminus X_{1+i_{j+1}}$ and $\cadet(\cadet (u)) \in X_{1+i_{j+1}}$. (Note that this is possible since $X_{i_{j}}$ reaches $X_{i_{j+1}}$.) Define the last node of $Y_r$ to be $v_k$.

Now that we have identified our remaining $\mathbf{S}$-boxings with certain sequences of maximal $\mathbf{S}$-cadet sequences, we wish to find a sign-reversing involution on all such sequences except the one generated by the algorithm. For such a sequence let ($1_j$) denote the condition that $X_{i_{j+2}}$ is not reached by $X_{i_{j}}$, and let ($2_j$) denote the condition that $X_{i_{j+1}}$ is the first $\mathbf{S}$-box of the full sequence reached by $X_{i_j}$, but not by $X_{i_{j-1}}$.
    
   Suppose there is some smallest index $q$ for which either $(1_q)$ is false or $(2_q)$ is false. If $(2_q)$ is false, we can insert $X_{p}$ after $X_{i_j}$ such that $X_p$ is the smallest indexed $\mathbf{S}$-box reached by $X_{i_j}$ but not $X_{i_{j-1}}$. This increases the length by 1 and makes $(1_q)$ false and $(2_q)$ true.
   
   Otherwise, $(2_q)$ is true and $(1_q)$ is false. We can then delete $X_{q+1}$ from the subsequence. This decreases the length by 1 and makes $(2_q)$ false. As these two operations are inverse to each other, we can consider only $\S$-boxings $B \in \mathcal{B}$ corresponding to sequences satisfying $(1_j)$ and $(2_j)$ for all~$j$. This leaves only the sequence generated by the algorithm.
\end{proof}

We observe that Theorem \ref{thm:characterization}, together with Lemmas \ref{lem:contributionProduct} and \ref{lem:contributionProduct2}, provide an algorithm for computing the contribution $r_{\mathbf{S}}(T)$ of any tree $T \in \T^{(m)}(n)$.

\begin{rem}\label{rem:complexity}
We note that the unsimplified version of Bernardi's formula can only be computed in exponential time. (This essentially follows from needing to find all possible $\mathbf{S}$-boxings of the given tree.) On the other hand, the complexity of our proposed algorithm appears to be at most 
$\mathcal{O}(n^2\log m)$. Indeed, the step which appears to be the most nontrivial is finding all of the maximal $\mathbf{S}$-cadet sequences. To do this, let us suppose we are starting with a cadet sequence $(v_1,\ldots,v_k)$. Then for each $i$, we denote by $f(i)$ the smallest index for which $\sum_{\ell = i+1}^{f(i)} \lsib(v_\ell) \in S_{i,f(i)}^-$. (If none exists, we can set $f(i) = k+1$.) Assuming a complexity of $\mathcal{O}(\log m)$ to determine if a number is in a given set $S_{i,\ell}^-$, we can compute the function $f$ in $\mathcal{O}(n^2\log m)$. The longest $\mathbf{S}$-cadet sequence starting at $v_i$ will then end at the vertex with index $g(i):=\min_{i \leq \ell \leq f(i)-1}(f(\ell)-1)$. We see that $g$ can be computed in $\mathcal{O}(n^2)$. It then remains to delete those $\mathbf{S}$-sequences we have generated which are not maximal (it is possible they can be ``extended downards''). This process can be done in $\mathcal{O}(n)$. Putting this together, we conclude that finding all of the maximal $\mathcal{S}$-cadet sequences of a given cadet sequence has complexity at most $\mathcal{O}(n^2\log m)$.
\end{rem}

In particular, the following corollaries make it easier to compute the contribution of many trees.

	\begin{cor}\label{cor:mustInt}
	In a tree with nonzero contribution, any maximal $\mathbf{S}$-cadet sequence of size more than one must intersect some other maximal $\mathbf{S}$-cadet sequence.
	\end{cor}
	
	\begin{proof}
	    Let $X$ be a maximal $\mathbf{S}$-cadet sequence which does not intersect some other maximal $\mathbf{S}$-cadet sequence. We observe that $X$ is an $\mathbf{S}$-connected cadet sequence. Moreover, if there exist two nodes $v, \cadet(v) \in X$, then as in the proof of the theorem, we can construct a sign-reversing involution by combining or splitting the $\mathbf{S}$-boxes containing $v$ and $\cadet(v)$. Thus any such tree has contribution~0.
	\end{proof}
	
\begin{rem}\label{rem:transitiveOnly}
    We observe that for transitive arrangements, no two distinct maximal $\mathbf{S}$-cadet sequences intersect. Thus, in this case, $\mathbf{S}$-connected cadet sequences and maximal $\mathbf{S}$-cadet sequences are one and the same.  This implies that the contribution of each $\mathbf{S}$-connected cadet sequence (and by extension each tree) is either 0 or 1. (This fact was leveraged by Bernardi to prove Theorem~\ref{thm:transitiveFormula}. Indeed, in \cite[Definition~4.5]{bernardi}, the set $\mathcal{T}_\mathbf{S}$ in the statement of Theorem~\ref{thm:transitiveFormula} is defined to be the set of trees $T \in \mathcal{T}^{(m)}(n)$ for which every maximal $\mathbf{S}$-cadet sequence consists of exactly one node.)
	On the other hand, we suspect that trees contributing -1 always exist for non-transitive arrangements. For example, if $s \notin S_{i,j}^-$, $t \notin S_{j,k}^-$, and $s + t \in S_{i,k}^-$, then an $\mathbf{S}$-connected cadet sequence $(v_i,v_j,v_k)$ with $\lsib(v_j) = s$ and $\lsib(v_k) = t$ would contribute $-1$.
\end{rem}

\begin{cor}\label{cor:firstLastSmall}
    Let $T \in \mathcal{T}^{(m)}(n)$ and let $X = (v_1,v_2,\ldots,v_k)$ be an $\mathbf{S}$-connected cadet sequence of $T$ with maximal $\mathbf{S}$-cadet sequences $X_1,\ldots,X_{k'}$ (in increasing order of index of last node). Define $X_0 = \{0\} = X_{k'+1}$. If $X$ has nonzero contribution and $k' > 1$, then $|X_1\setminus X_2| = 1 = |X_{k'}\setminus X_{k'-1}|$.
\end{cor}

\begin{proof}
    The fact that $|X_1\setminus X_2| = 1$ is immediate from Remark \ref{rem:empty} and Theorem \ref{thm:characterization}. Moreover, by Theorem \ref{thm:characterization}, there exists $j < k'$ so that $X_j$ reaches $X_{k'}$. Since the largest indexed node in $X_{k'}\setminus X_{k'+1}$ is $v_k$, this means $v_{k-1} \in X_j$ and $j = k'-1$.
\end{proof}
	
	\section{Contributions for Almost Transitive Arrangements}\label{sec:Ish}
	In this section, we define \textit{almost transitive} arrangements and show that Ish-type arrangements are almost transitive. We then further characterize the contributions of trees for almost transitive arrangements.
	
	\begin{defn}\label{def:Almost Transitive}
	We say a hyperplane arrangement $\mathcal{A}_\mathbf{S}$ is \emph{almost transitive} if for all distinct $i,j,k \in [n]$ with $1 \notin \{i,k\}$ and for all nonnegative integers $s \in S_{i,j}^-$ and $t \in S_{j,k}^-$, we have $s + t \notin S_{i,k}^{-}$.
	\end{defn}
	
	\begin{rem}
	Since permuting the axes does not affect the number of regions of a hyperplane arrangement or the corresponding signed sum of trees, the condition $1 \notin \{i,k\}$ could be replaced with $c \notin \{i,k\}$ for any fixed $c \in [n]$.
	\end{rem}
	
	\begin{lem}\label{lem:IshAlmost Transitive}
	{ Any Ish-type arrangement is almost transitive}.
	\end{lem}
	
	\begin{proof}
	Let $\mathcal{A}_\mathbf{S}$ be an Ish-type arrangement. We observe that $S^{-}_{i,j} = \{0\}$ whenever $1 \notin\{i,k\}$. Now let $i,j,k \in [n]$ be distinct with $1 \notin \{i,k\}$ and let $s \notin S_{i,j}^-$ and $t \notin S_{j,k}^-$ be nonnegative integers. Note that $0 \in S_{i,j}^-\cap S_{j,k}^-$, so we must have $s > 0$ and $t > 0$. This means $s + t \notin S_{i,k}^- = \{0\}$, as desired.
	\end{proof}
	
	\begin{lem}\label{lem_extendSequence}
	{ Let $\mathcal{A}_\mathbf{S}$ be almost transitive and let $T \in \mathcal{T}^{(m)}(n)$. Suppose that $(v_1,v_2,\ldots,v_k)$ and $(v_k,v_{k+1},\ldots,v_{k+\ell})$ are $\mathbf{S}$-cadet sequences of $T$. If either $v_k = 1$ or neither sequence contains the node 1, then $(v_1,v_2,\ldots,v_{k+\ell})$ is an $\mathbf{S}$-cadet sequence of $T$.}
	\end{lem}

	\begin{proof}
	Let $i < k$ and $0 <j \leq \ell$. Then by assumption $\sum_{p = i+1}^k \lsib(v_p) \notin S_{v_i,v_k}^-$ and $\sum_{p = k+1}^{k+j}\lsib(v_p) \notin S_{v_k,v_{k+j}}^-$. Now neither $v_i$ nor $v_{k+j}$ is the node 1. Thus since $\mathcal{A}_\mathbf{S}$ is almost transitive, we have $\sum_{p = i+1}^{k+j}\lsib(v_p) \notin S_{v_i,v_{k+j}}^-$. This implies the result.
	\end{proof}

	\begin{cor}\label{cor:boxesIntersect}
	    Let $\mathcal{A}_\mathbf{S}$ be almost transitive and let $T \in \mathcal{T}^{(m)}(n)$. Suppose there exist positive integers $a,b,c$ and nodes $u_1,\ldots,u_{a+b+c}$ such that both $(u_1,u_2,\ldots,u_{a+b})$ and $(u_a,u_{a+1},\ldots,u_{a+b+c})$ are maximal $\mathbf{S}$-cadet sequences of $T$. Then $1 \in \{u_{a-1},u_{a+b+1}\}$. 
	\end{cor}
	
	\begin{proof}
	Since $(u_{a-1},u_a,\ldots,u_{a+b+c})$ is not an $\mathbf{S}$-cadet sequence, Lemma \ref{lem_extendSequence} implies that either $u_{a-1} =1$ or there exists $a < j \leq a+b+c$ such that $u_j = 1$. If $u_{a-1} = 1$ we are done, so suppose there exists such a $j$. Since $(u_1,\ldots,u_{a+b},u_{a+b+1})$ is not an $\mathbf{S}$-cadet sequence, Lemma \ref{lem_extendSequence} then implies that $j = a+b+1$.
	\end{proof}
	
	Now, we can fully characterize the contribution of a tree.
	
	\begin{prop}\label{lem:IshContribution}
	    	Let $\mathcal{A}_\mathbf{S}$ be almost transitive and let $T \in \mathcal{T}^{(m)}(n)$. Then
	the contribution $r_\mathbf{S}(T)$ is nonzero if and only if the following conditions hold:
	\begin{enumerate}
		\item Any maximal cadet sequence of $T$ not containing the node $1$ does not contain any maximal $\mathbf{S}$-cadet sequences of size greater than $1$.
		
		\item In the maximal cadet sequence $(v_1,v_2,\ldots,v_t)$ containing the node $1 =: v_j$,  there exist unique integers $0 \leq i < j$ and $0 \leq k < t-j$ so that:
		\begin{enumerate}
		    \item If $t < j-i-1$ or $t > j + k + 1$, then $\{v_t\}$ is a maximal $\mathbf{S}$-cadet sequence.
		    \item There is a maximal $\mathbf{S}$-cadet sequence $\{v_{j-i},\ldots,v_j,\ldots,v_{j+k}\}$.
		    \item If $j-i \neq 1$, then there is a maximal $\mathbf{S}$-cadet sequence $\{v_{j-i-1}, v_{j-i},\ldots,v_{j-1}\}$. If $j-i = 1$, then $i = 0$.
		    \item If $j+k \neq t$, then there is a maximal $\mathbf{S}$-cadet sequence
		    $\{v_{j+1},\ldots,v_{j+k},v_{j+k+1}\}$. If $j + k = t$, then $k = 0$.
		    \item There are no other maximal $\mathbf{S}$-cadet sequences.
		\end{enumerate}
	\end{enumerate}
			Furthermore, if $T$ satisfies these conditions, then its contribution is given by $r_\mathbf{S}(T) = (-1)^{k+i}$.
	\end{prop}
	
	\begin{proof}
	Let $T \in \mathcal{T}^{(m)}(n)$ have nonzero contribution. Then each of its maximal cadet sequences must contribute $\pm 1$ by Lemma \ref{lem:contributionProduct}. By Corollary \ref{cor:mustInt}, this means any maximal $\mathbf{S}$-cadet sequence of size greater than 1 intersects another maximal $\mathbf{S}$-cadet sequence, and so condition (1) is a direct consequence of Corollary \ref{cor:boxesIntersect}.
	
	Now consider the maximal cadet sequence $X = (v_1,v_2,\ldots,v_t)$ containing $1 =: v_j$. Since we have already shown (1), we note that the contribution of $X$ is equal to the contribution of $T$.
	
	Suppose $X$ has maximal $\mathbf{S}$-cadet sequences $(X_1,X_2,\ldots,X_{k'})$ in increasing order of index of last node. As usual, define $X_0 = \{0\} = X_{k'+1}$. Since $X$ is the union of its maximal $\mathbf{S}$-cadet sequences, there exists some maximal $\mathbf{S}$-cadet sequence $X_r$ which can be written $X_r = \{v_{j-i},\ldots,v_j,\ldots,v_{j+k}\}$ for some $0 \leq i \leq j$ and $0 \leq k \leq t-k$. Thus the tree $T$ satisfies (2b). Now by Corollary \ref{cor:boxesIntersect}, we see that if a maximal $\mathbf{S}$-cadet sequence $X_s$ intersects $X_r$, then either the last node of $X_s$ is $v_{j-1}$ or the first node of $X_s$ is $v_{j+1}$. This means $X_{r-1}$ intersects $X_r$ if and only if either $X_{r-1}$ has size greater than 1$, X_{r+1}$ intersects $X_r$ if and only if $X_{r+1}$ has size greater than 1, and no other maximal $\mathbf{S}$-cadet sequence intersects $X_r$. Moreover, if $s < r-1$ or $s > r+1$, then $X_{s}$ does not intersect any other maximal $\mathbf{S}$-cadet sequence, and thus has size one. We have therefore shown that $T$ satisfies~(2a).
	
	If $j-i = 1$ (so that $r = 1$), then Corollary \ref{cor:firstLastSmall} implies that $X_{r+1}\setminus X_r = \{v_j\}$ and hence $i = 0$. Likewise, if $j+k = t$ (so that $r = k'$), then Corollary \ref{cor:firstLastSmall} implies that $X_{r}\setminus X_{r-1} = \{v_j\}$ and hence $k = 0.$
	
	We now show that if $r \neq 1$ then the first node of $X_{r-1}$ is $v_{j-i-1}$ and if $r \neq t$ then the last node of $X_{r+1}$ is $v_{j + k + 1}$. Suppose $r \neq 1$ and observe that, for any $s$, if $X_s$ and $X_{s+1}$ both have size 1, then $X_s$ reaches $X_{s+1}$ and does not reach any other maximal $\mathbf{S}$-cadet sequence. Thus by Theorem \ref{thm:characterization}, we have a sequence $(X_0,X_1,\ldots,X_{r-2},X_{r-1})$ where each maximal $\mathbf{S}$-cadet sequence reaches the next. Thus the largest indexed node in $X_{r-1}\setminus X_r$ is $v_{i-j-1}$ as desired. The result for $X_{r+1}$ is completely analogous, observing that $X_{r}$ must reach $X_{r+1}$.
	
	We have thus shown (2c) and (2d). Condition (2e) then follows immediately from the observation that the only possible maximal $\mathbf{S}$-cadet sequences of size greater than 1 are those in (2b), (2c), and~(2d).
	
	Now suppose $T$ satisfies (1) and (2). We then see that the contribution of $T$ is equal to the contribution of the maximal cadet sequence containing 1 and that this contribution is nonzero. Indeed, using the notation as above we see that $(X_0,X_1,\ldots,X_{k'})$ is a valid output of the algorithm in Theorem \ref{thm:characterization}. Furthermore, we have that $k' = (j-i-1) + 1 + (t-k-j)$. Thus by Theorem \ref{thm:characterization}, the contribution of $T$ is
	$$r_{\mathbf{S}}(T) = (-1)^{t-(j-i+t-k-j)} = (-1)^{i + k}.$$
	\end{proof}
	
	    \begin{defn}\label{not:ell}
	        Let $\mathcal{A}_\mathbf{S}$ be almost transitive and let $T \in \mathcal{T}^{(m)}(n)$ have nonzero contribution. Let $i$ and $k$ be as in the statement of Proposition \ref{lem:IshContribution}. Then we refer to $L_l(T) := i$ and $L_u(T):=k$ as the \emph{lower 1-length} and \emph{upper 1-length} of $T$, respectively.
	    \end{defn}
	    
	    As an example, the tree shown in Figure \ref{fig:inefficient} has $L_l(T) = 1$ and $L_u(T) = 0$. (See Notation \ref{not:s} and Example \ref{ex:inefficient} below for an explanation of the caption.)  We note the the lower and upper 1-length of a (rooted labeled plane) tree depend on the hyperplane arrangement $\mathcal{A}_\mathbf{S}$, even though this is not reflected in our notation.

	\section{1-length and Inefficiency}\label{sec:IshParams}
	In this section, we work exclusively with Ish-type arrangements. We recall from Lemma \ref{lem:IshAlmost Transitive} that these arrangements are almost transitive. We therefore study the conditions outlined in Proposition \ref{lem:IshContribution} in detail. We do so using the notions of upper and lower 1-length and two additional parameters, which we define now.
	
	\begin{defn}\label{def:inefficient}
	    Let $\mathcal{A}_\mathbf{S}$ be an Ish-type arrangement and Let $T \in \mathcal{T}^{(m)}(n)$ and let $(v_1,\ldots,v_j,\ldots,v_t)$ be the maximal cadet sequence of $T$ which contains the node $1=:v_j$. We say a node $w$ of $T$ is \emph{lower inefficient} if $w$ is a left sibling of 1 and $\lsib(w) \notin S_{w,1}^-$. Likewise, we say $w$ is \emph{upper inefficient} if $\parent(w) = 1$, $w$ is not the cadet of 1, and $\lsib(w) \notin S_{1,w}^-$. We denote by $E_l(T)$ and $E_u(T)$ the number of lower and upper inefficient nodes in $T$, respectively. We refer to these values as the \emph{lower inefficiency} and \emph{upper inefficiency} of $T$.
	\end{defn}
	
We emphasize that, as with lower and upper 1-length, the lower and upper inefficiency of a (rooted labeled plane) tree depend on the hyperplane arrangement $\mathcal{A}_\mathbf{S}$, even though this is not reflected in our notation.
	
	\begin{nota}\label{not:s}
	    We denote by $\mathbf{S}(e_l,\ell_l,e_u,\ell_u)$ the set of trees $T \in \mathcal{T}^{(m)}(n)$ with nonzero contribution so that $E_l(T) = e_l, L_l(T) = \ell_l$, $E_u(T) = e_u$, and $L_u(T) = \ell_u$. We note that $\mathbf{S}(e_l,\ell_l,e_u,\ell_u)$ is necessarily empty if $e_l+\ell_l+e_u+\ell_u \geq n$.
	\end{nota}
	
\begin{ex}\label{ex:inefficient}
	Let $\mathcal{A}_\mathbf{S}$ be the nested Ish arrangement with $n = 8$ and each $S_{1,j} = \{-j+1,\ldots,j-1\}$. Then the tree in Figure \ref{fig:inefficient} is an element of $\mathbf{S}(1,1,0,2)$. The maximal $\mathbf{S}$-cadet sequences of size larger than 1 are $\{7,5\}$ and $\{5,1\}$, the upper inefficient nodes are 2 and 3, and the lower inefficient node is 4.
\end{ex}

	\begin{figure}
	\centering
	\begin{tikzpicture}
		\draw (0,0)--(0,1);
		\draw (0,0)--(-1,1);
		\draw (0,1)--(1.75,2);
		\draw (0,1)--(1.25,2);
		\draw (0,1)--(0.5,2);
		\draw (0,1)--(-1.25,2);
		\draw (0,1)--(-0.5,2);
		\draw (0,1)--(0,2);
		\draw (1.75,2)--(1.75,3);
		\draw (1.75,2)--(2.5,3);
		\draw (1.75,2)--(3.25,3);
		\draw (1.75,2)--(1,3);
		\draw (1.75,2)--(0.5,3);
		\node[draw,circle,fill=white] at (0,0) {$7$};
	    \node[draw,circle,fill=white] at (0,1) {$5$};
	    \node[draw,circle,fill=white] at (1.75,3) {$2$};
	    \node[draw,circle,fill=white] at (0.5,2) {$4$};
	    \node[draw,circle,fill=white] at (1.75,2) {$1$};
	    \node[draw,circle,fill=white] at (2.5,3) {$3$};
	    \node[draw,circle,fill=white] at (3.25,3) {$6$};
	    \node[draw,circle,fill=white] at (-0.5,2) {$8$};
	\end{tikzpicture}
	\caption{A tree in $\mathbf{S}(1,1,0,2)$, where $\mathcal{A}_\mathbf{S}$ is the nested Ish arrangement with $n = 8$ and each $S_{1,j} = \{-j+1,\ldots,j-1\}$.}\label{fig:inefficient}.
	\end{figure}
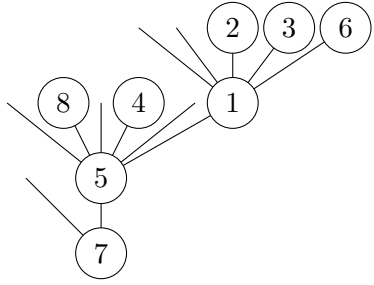
	
    We now describe a series of sign-reversing involutions in order to reduce Bernardi's formula to the enumeration of the trees in $\mathbf{S}(0,0,0,0)$. In particular, we will show that  the trees with lower 1-length $\ell_l \neq 0$ (resp. upper 1-length $\ell_u \neq 0$) and lower (resp. upper) inefficiency 0 are in bijection with those of lower  1-length $\ell_l - 1$ (resp. upper 1-length $\ell_u-1$) and nonzero lower (resp. upper) inefficiency.
    
    We first fix some $e_l,\ell_l,e_u,\ell_u$ with $\ell_l \neq 0$ and let $T \in \mathbf{S}(e_l,\ell_l,e_u,\ell_u)$. We denote $(v_1,\ldots,v_j,\ldots,v_t)$ the maximal cadet sequence of $T$ containing the node $1=: v_j$. Note that by Proposition \ref{lem:IshContribution}, the fact that $L_l(T) = \ell_l \neq 0$ implies that $j - \ell_l \neq 1$. We now construct a new (rooted labeled plane) tree $\phi_l(T)$ as follows:
    \begin{enumerate}
        \item Delete all of the right siblings of the node $v_{j-1}$ (which are necessarily leaves), and denote by $s_1,s_2,\ldots,s_{\lsib(1)}$ the left siblings of the node 1 (including leaves), indexed left to right.
        \item Let $c$ be the leftmost child of $v_{j-2}$. (Note that $c$ may or may not be a node). Replace the edge from $v_{j-2}$ to $c$ with a pair of edges from $v_{j-2}$ to a new node $v_{j-1}'$ and from $v_{j-1}'$ to $c$.
        \item For each $i \in [\lsib(1)]$ in order, delete the edge from $v_{j-1}$ to $s_i$ and draw an edge from $v_{j-2}$ to $s_i$ so that $s_i$ has $i-1$ left siblings.
        \item Delete the node $v_{j-1}$ and all incident edges, relabel $v_{j-1}'$ by $v_{j-1}$, and draw an edge from $v_{j-2}$ to $1$ so that $1$ is the rightmost child of $v_{j-2}$.
        \item Add and delete leaves as rightmost children where necessary so that every node has $m+1$ children.
    \end{enumerate}
    The idea behind this construction is to move $v_{j-1}$ from being in a maximal $\mathbf{S}$-cadet sequence with 1 to being a lower inefficient node.
    An example is shown in Figure \ref{ex:involution1}.

	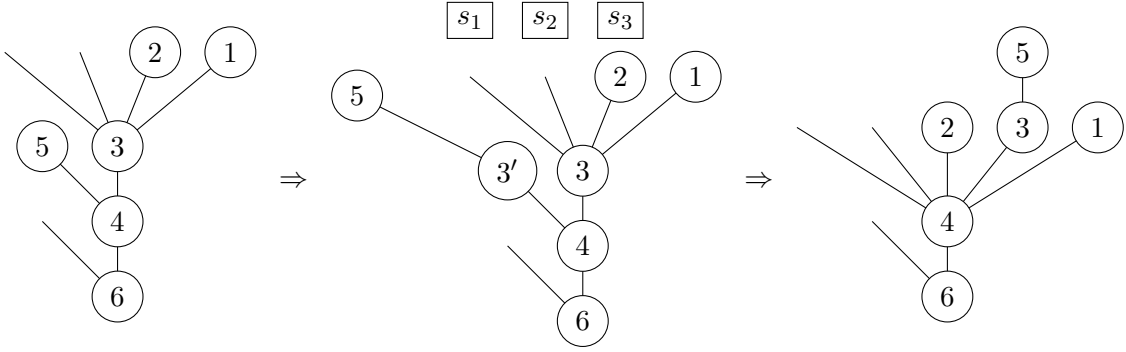
\begin{figure}
	\centering
   	$\begin{array}{l}\begin{tikzpicture}
		\draw (0,0)--(0,1);
		\draw (0,0)--(-1,1);
		\draw (0,1)--(0,2);
		\draw (0,1)--(-1,2);
		\draw (0,2)--(-1.5,3.25);	\draw (0,2)--(-0.5,3.25);
		\draw (0,2)--(0.5,3.25);
		\draw (0,2)--(1.5,3.25);
		\node[draw,circle,fill=white] at (0,0) {$6$};
		\node[draw,circle,fill=white] at (0,1) {$4$};
		\node[draw,circle,fill=white] at (-1,2) {$5$};
		\node[draw,circle,fill=white] at (0,2) {$3$};
		\node[draw,circle,fill=white] at (0.5,3.25) {$2$};
		\node[draw,circle,fill=white] at (1.5,3.25) {$1$};
	\end{tikzpicture}\end{array}$
	$\Rightarrow$
		$\begin{array}{l}\begin{tikzpicture}
		\draw (0,0)--(0,1);
		\draw (0,0)--(-1,1);
		\draw (0,1)--(0,2);
		\draw (0,1)--(-1,2);
		\draw (0,2)--(-1.5,3.25);	\draw (0,2)--(-0.5,3.25);
		\draw (0,2)--(0.5,3.25);
		\draw (0,2)--(1.5,3.25);
		\draw (-1,2)--(-3,3);
		\node[draw,circle,fill=white] at (0,0) {$6$};
		\node[draw,circle,fill=white] at (0,1) {$4$};
		\node[draw,circle,fill=white] at (-1,2) {$3'$};
		\node[draw,circle,fill=white] at (-3,3) {$5$};
		\node[draw,circle,fill=white] at (0,2) {$3$};
		\node[draw,circle,fill=white] at (0.5,3.25) {$2$};
		\node[draw,circle,fill=white] at (1.5,3.25) {$1$};
		\node[draw] at (-1.5,4) {$s_1$};
	    \node[draw] at (-0.5,4) {$s_2$};
	    \node[draw] at (0.5,4) {$s_3$};
	\end{tikzpicture}\end{array}$
	$\Rightarrow$
	$\begin{array}{l}	\begin{tikzpicture}
		\draw (0,0)--(0,1);
		\draw (0,1)--(2,2.25);
		\draw (0,1)--(1,2.25);
		\draw (0,1)--(0,2.25);
		\draw (0,1)--(-1,2.25);
        \draw (1,3.25)--(1,2.25);
        \draw (0,0)--(-1,1);
        \draw (0,1)--(-2,2.25);
		\node[draw,circle,fill=white] at (0,0) {$6$};
	    \node[draw,circle,fill=white] at (0,1) {$4$};
	    \node[draw,circle,fill=white] at (1,2.25) {$3$};
	    \node[draw,circle,fill=white] at (1,3.25) {$5$};
	    \node[draw,circle,fill=white] at (2,2.25) {$1$};
	    \node[draw,circle,fill=white] at (0,2.25) {$2$};
	\end{tikzpicture}\end{array}$
	\caption{An example of the map $\phi_l$ applied to a tree in $\mathbf{S}(1,2,0,0)$, where $\mathcal{A}_\mathbf{S}$ is the Ish arrangement with $n = 6$. All leaves which do not have any nodes as left siblings have been omitted for clarity. We note that the final tree is an element of $\mathbf{S}(2,1,0,0)$.}\label{ex:involution1}
	\end{figure}

    \begin{lem}\label{lem:invo1}
        Let $T \in \mathbf{S}(e_l,\ell_l,e_u,\ell_u)$. Then $\phi_l(T)\in \bigcup_{b = e_l+1}^{n-1} \mathbf{S}(b,\ell_l-1,e_u,\ell_u).$
    \end{lem}
    
    \begin{proof}
        We first claim that $\phi_l(T)$ is a well-defined member of $\mathcal{T}^{(m)}(n)$. Indeed, we have $\lsib_{\phi_l(T)}(1) = \lsib_T(1) + \lsib_T(v_{j-1})$, $\lsib_{\phi_l(T)}(v_{j-1}) = \lsib_T(1)$, and $\lsib_{\phi_l(T)}(w) = \lsib_T(w)$ for all other nodes $w$. Now recall from Proposition~\ref{lem:IshContribution}(2) that both $\{v_{j-\ell_l-1},\ldots,v_{j-1}\}$ and $\{v_{j-\ell},\ldots,v_j\}$ are $\mathbf{S}$-cadet sequences, but $\{v_{j-\ell_l-1},\ldots,v_j\}$ is not. This means
        $$\sum_{p = j-\ell_l}^j \lsib_T(v_p) \in S_{v_{j-\ell_l-1},1}^-.$$
       This sum gives an upper bound on $\lsib_T(1) + \lsib_T(v_{j-1})$, which means 1 (and hence every node) has at most $m$ left siblings in $\phi_l(T)$. We conclude that $\phi_l(T) \in \mathcal{T}^{(m)}(n)$, as desired.
       
       Now note that $v_{j-1}$ is lower inefficient in $\phi_l(T)$ since $\{v_{j-1},v_j\}$ is an $\mathbf{S}$-cadet sequence of $T$. Moreover, if a left sibling $s_i$ of $1$ in $T$ is a node, then $s_i$ is still a left sibling of 1 in $\phi_l(T)$ with $\lsib_{\phi_l(T)}(s_i) = \lsib_{T}(s_i)$. This means $s_i$ is lower inefficient in $T$ if and only if it is lower inefficient in $\phi_l(T)$. We conclude that $E_l(\phi_l(T)) > E_l(T) = e_l$.
       
       We now observe that $(v_1,\ldots,v_{j-2},v_j,\ldots,v_t)$ is a maximal cadet sequence of $\phi_l(T)$. It is straightforward to verify that the maximal $\mathbf{S}$-cadet sequences of this cadet sequence are precisely those of $(v_1,\ldots,v_t)$ in $T$ with the node $v_{j-1}$ removed. This means $L_l(\phi_l(T)) = L_l(T) - 1 = \ell_l-1$ and $L_u(\phi_l(T) = L_u(T) = \ell_u$. Moreover, if $j+\ell_u \neq t$, then the left siblings of $v_{j+\ell_u+1}$ are the same in both $T$ and $\phi_l(T)$. This implies that $E_u(\phi_l(T)) = E_u(T) = e_u$.
       
       It remains to show that no other maximal cadet sequence of $\phi_l(T)$ contains a maximal $\mathbf{S}$-cadet sequence of size larger than 1. Indeed, consider the maximal cadet sequence of $\phi_l(T)$ which contains $v_{j-1}$. We note that $v_{j-1}$ is a left sibling of 1, and hence is not the cadet of any node. Moreover, if $v_{j-1}$ has a cadet, then this cadet must be $c$ and we have $\lsib_{\phi_l(T)}(c) = 0 \in S_{v_{j-1},c}^-$ by construction. As no other nodes have their cadets or number of left siblings changed in moving from $T$ to $\phi_l(T)$, this implies the result.
    \end{proof}
    
Now let $T \in \cup_{b = e_l+1}^{n-1} \mathbf{S}(b,\ell_l-1,e_u,\ell_u)$, and let $(v_1,\ldots,v_j,\ldots,v_t)$ be the maximal cadet sequence of $T$ containing the node $1=:v_j$. Since $E_l(T) \neq 0$, it must be that $j \neq 1$. Now construct a new tree $\psi_l^{e_l}(T)$ as follows:
\begin{enumerate}
    \item Let $w$ be the $(e_l + 1)$-th lower inefficient node of $T$ (from the left). Denote by $s_1,s_2,\ldots,s_{\lsib(w)}$ the left siblings of $w$, indexed left to right.
    \item Let $c$ be the leftmost child of $w$. (Note that $c$ may or may not be a node.) Delete all right siblings of $c$ (necessarily leaves), and replace the edges from $v_{j-1}$ to $w$ and from $w$ to $c$ with an edge from $v_{j-1}$ to $c$.
    \item Delete the node $w$ and all incident edges and replace the edge from $v_{j-1}$ to 1 with edges from $v_{j-1}$ to a new node again labeled $w$ and this new node $w$ to 1.
    \item Delete the edges from $v_{j-1}$ to $s_1,\ldots,s_{\lsib(w)}$ and draw edges from $w$ to $s_1,\ldots,s_{\lsib(w)}$ so that the left siblings of 1 are $s_1,\ldots,s_{\lsib(w)}$ in that order.
    \item Add and delete leaves as rightmost children where necessary so that every node has $m+1$ children.
\end{enumerate}

An example of this construction is shown in Figure \ref{ex:involution2}. We emphasize that the parameter $e_l$ must be specified in the definition of $\psi_l^{e_l}$. Indeed, this will allow us to accurately describe the bijections in our proof of Theorem \ref{thm:involution}.

        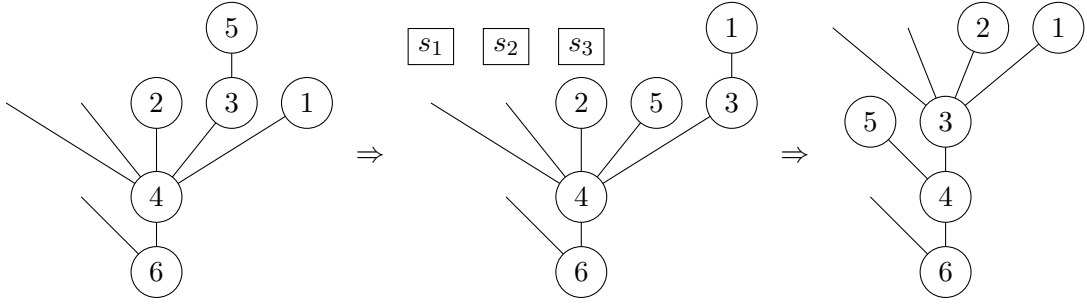
\begin{figure}
        \centering
            $\begin{array}{l}	\begin{tikzpicture}
		\draw (0,0)--(0,1);
		\draw (0,1)--(2,2.25);
		\draw (0,1)--(1,2.25);
		\draw (0,1)--(0,2.25);
		\draw (0,1)--(-1,2.25);
        \draw (1,3.25)--(1,2.25);
        \draw (0,0)--(-1,1);
        \draw (0,1)--(-2,2.25);
		\node[draw,circle,fill=white] at (0,0) {$6$};
	    \node[draw,circle,fill=white] at (0,1) {$4$};
	    \node[draw,circle,fill=white] at (1,2.25) {$3$};
	    \node[draw,circle,fill=white] at (1,3.25) {$5$};
	    \node[draw,circle,fill=white] at (2,2.25) {$1$};
	    \node[draw,circle,fill=white] at (0,2.25) {$2$};
	\end{tikzpicture}\end{array}$
	$\Rightarrow$ 
	$\begin{array}{l}	\begin{tikzpicture}
		\draw (0,0)--(0,1);
		\draw (0,1)--(2,2.25);
		\draw (0,1)--(1,2.25);
		\draw (0,1)--(0,2.25);
		\draw (0,1)--(-1,2.25);
        \draw (2,3.25)--(2,2.25);
        \draw (0,0)--(-1,1);
        \draw (0,1)--(-2,2.25);
		\node[draw,circle,fill=white] at (0,0) {$6$};
	    \node[draw,circle,fill=white] at (0,1) {$4$};
	    \node[draw,circle,fill=white] at (1,2.25) {$5$};
	    \node[draw,circle,fill=white] at (2,3.25) {$1$};
	    \node[draw,circle,fill=white] at (2,2.25) {$3$};
	    \node[draw,circle,fill=white] at (0,2.25) {$2$};
	    \node[draw] at (-2,3) {$s_1$};
	    \node[draw] at (-1,3) {$s_2$};
	    \node[draw] at (0,3) {$s_3$};
	\end{tikzpicture}\end{array}$
	$\Rightarrow$ $\begin{array}{l}\begin{tikzpicture}
		\draw (0,0)--(0,1);
		\draw (0,0)--(-1,1);
		\draw (0,1)--(0,2);
		\draw (0,1)--(-1,2);
		\draw (0,2)--(-1.5,3.25);	\draw (0,2)--(-0.5,3.25);
		\draw (0,2)--(0.5,3.25);
		\draw (0,2)--(1.5,3.25);
		\node[draw,circle,fill=white] at (0,0) {$6$};
		\node[draw,circle,fill=white] at (0,1) {$4$};
		\node[draw,circle,fill=white] at (-1,2) {$5$};
		\node[draw,circle,fill=white] at (0,2) {$3$};
		\node[draw,circle,fill=white] at (0.5,3.25) {$2$};
		\node[draw,circle,fill=white] at (1.5,3.25) {$1$};
	\end{tikzpicture}\end{array}$
	\caption{An example of the map $\psi_l^1$ applied to a tree in $\mathbf{S}(2,1,0,0)$, where $\mathcal{A}_\mathbf{S}$ is the Ish arrangement with $n = 6$. All leaves which do not have any nodes as left siblings have been omitted for clarity. We note that this gives an inverse to the construction in Figure~\ref{ex:involution1}.}\label{ex:involution2}
	\end{figure}

\begin{lem}\label{lem:invo2}
    Let $T \in \bigcup_{b = e_l+1}^{n-1} \mathbf{S}(b,\ell_l-1,e_u,\ell_u)$. Then $\psi_l^{e_l}(T)\in\mathbf{S}(e_l,\ell_l,e_u,\ell_u)$.
\end{lem}

\begin{proof}
    We first claim that $\psi_l^{e_l}(T)$ is a well-defined member of $\mathcal{T}^{(m)}(n)$. Indeed, we have that $\lsib_{\psi_l^{e_l}(T)}(1) = \lsib_T(w)$, $\lsib_{\psi_l^{e_l}(T)}(w) = \lsib_T(1) - \lsib_T(w)$, and $\lsib_{\psi_T}(w') = \lsib_T(w')$ for all other nodes $w'$. We conclude that every node has at most $m$ left siblings in $\psi_l^{e_l}(T)$ and thus $\psi_l^{e_l}(T) \in \mathcal{T}^{(m)}(n)$.
    
     We next observe that $(v_1,\ldots,v_{j-1},w,v_j,\ldots,v_t)$ is a maximal cadet sequence of $\psi_l^{e_l}(T)$. It is straightforward to show that the maximal $\mathbf{S}$-cadet sequences of this cadet sequence are precisely those of $(v_1,\ldots,v_j,\ldots,v_t)$ in $T$ with $u$ appended to the maximal $\mathbf{S}$-cadet sequence which ends with $v_{j-1}$ and to the maximal $\mathbf{S}$-cadet sequence which contains $1 = v_j$. This means $L_l(\psi_l^{e_l}(T)) = L_l(T)+1 = \ell_l$ and $L_u(\psi_l^{e_l}(T)) = L_u(T) = \ell_u$. Moreover, if $j+\ell_u \neq t$, then the left siblings of $v_{j+\ell_u+1}$ are the same in both $T$ and $\psi_l^{e_l}(T)$. This implies that $E_u(\psi_l^{e_l}(T)) = E_u(T) = e_u$.
    
    Now, since no other nodes have their cadets or number of left siblings changed in moving from $T$ to $\psi_l^{e_l}(T)$, all other maximal $\mathbf{S}$-cadet sequences of $\psi_l^{e_l}(T)$ have size 1. Furthermore, the left siblings of 1 in $\psi_l^{e_l}(T)$ are precisely those of $u$ in $T$ (in the same order). In particular, the lower inefficient nodes in $\psi_l^{e_l}(T)$ correspond precisely with those lower inefficient nodes of $T$ which are left of $v$, of which there are $e_l$. This implies the result.
\end{proof}

We now give ``upper analogues'' to the constructions $\phi_l$ and $\psi_l^{e_l}$. We remark that the proofs of
Lemmas \ref{lem:invo3} and \ref{lem:invo4} below are similar to those of Lemmas \ref{lem:invo1} and \ref{lem:invo2}.

Again fix some $e_l,\ell_l,e_u,\ell_l$, now with $\ell_u \neq 0$, and let $T \in \mathbf{S}(e_l,\ell_l,e_u,\ell_u)$. We denote by $(v_1,\ldots,v_j,\ldots,v_t)$ the maximal cadet sequence of $T$ containing $1=:v_j$. Note that by Proposition \ref{lem:IshContribution}, the fact that $L_u(T) = \ell_u \neq 0$ implies that $j+\ell_u \neq t$.
    We now construct a new tree $\phi_u(T)$ as follows:
    \begin{enumerate}
        \item Delete all right siblings of the node $v_{j+1}$ (necessarily leaves) and denote the left siblings of the node $v_{j+2}$ by $s_1,\ldots,s_{\lsib(v_{j+2})}$, indexed left to right.
        \item For each $i \in \{2,\ldots,\lsib(v_{j+2})\}$ in order, delete the edge from $v_{j+1}$ to $s_i$ and draw an edge from $1$ to $s_i$ so that $s_i$ has $\lsib_T(v_{j+1}) + i - 1$ left siblings. (That is, at each step $s_i$ is the rightmost child of 1.) Likewise, delete the edge from $v_{j+1}$ to $v_{j+2}$ and draw an edge from $1$ to $v_{j+2}$ so that $v_{j+2}$ is the rightmost child of 1.
        \item Add and delete leaves as rightmost children where necessary so that every node has $m+1$ children.
    \end{enumerate}
    The idea behind the construction is to move $v_{j+1}$ from being in a maximal $\mathbf{S}$-cadet sequence with 1 to being an upper inefficient node. An example is show in Figure \ref{ex:involution3}.

	\begin{figure}
	\centering
	   $\begin{array}{l}\begin{tikzpicture}
		\draw (0,0)--(-1.5,1);
		\draw (0,0)--(-0.5,1);
		\draw (0,0)--(0.5,1);
		\draw (0,0)--(1.5,1);
		\draw (1.5,1)--(0.5,2);
		\draw (1.5,1)--(1.5,2);
		\draw (1.5,1)--(2.5,2);
		\node[draw,circle,fill=white] at (0,0) {$1$};
		\node[draw,circle,fill=white] at (-1.5,1) {$3$};
		\node[draw,circle,fill=white] at (-0.5,1) {$2$};
		\node[draw,circle,fill=white] at (1.5,1) {$6$};
		\node[draw,circle,fill=white] at (0.5,2) {$4$};
		\node[draw,circle,fill=white] at (2.5,2) {$5$};
	\end{tikzpicture}\end{array}$
		\begin{tikzcd} \phantom{1}\arrow[r,"\textnormal{\large$\phi_u$}",yshift=0.1cm]&\phantom{1}\arrow[l,"\textnormal{\large$\psi_u^1$}",yshift=-0.1cm]\end{tikzcd}
  	$\begin{array}{l}\begin{tikzpicture}
		\draw (0,0)--(-2.5,1);
		\draw (0,0)--(-1.5,1);
		\draw (0,0)--(-0.5,1);
		\draw (0,0)--(0.5,1);
		\draw (0,0)--(1.5,1);
		\draw (0,0)--(2.5,1);
        \draw (0.5,1)--(0.5,2);
		\node[draw,circle,fill=white] at (0,0) {$1$};
		\node[draw,circle,fill=white] at (-2.5,1) {$3$};
		\node[draw,circle,fill=white] at (-1.25,1) {$2$};
		\node[draw,circle,fill=white] at (0.5,1) {$6$};
		\node[draw,circle,fill=white] at (0.5,2) {$4$};
		\node[draw,circle,fill=white] at (2.5,1) {$5$};
	\end{tikzpicture}\end{array}$
	\caption{An example of the maps $\phi_u$ and $\psi_u^1$, where $\mathcal{A}_\mathbf{S}$ is the nested Ish arrangement with $n = 6$ and $S_{1,j} = \{-j+1,\ldots,0\}$ for each $j$. The tree on the left is an element of $\mathbf{S}(0,0,1,1)$ and the tree on the right is an element of $\mathbf{S}(0,0,2,0)$.}\label{ex:involution3}
	\end{figure}
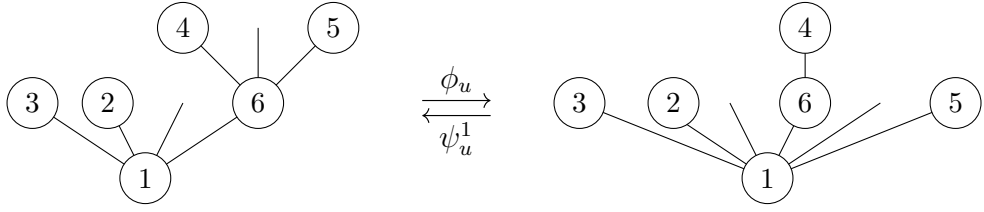

    \begin{lem}\label{lem:invo3}
        Let $T \in \mathbf{S}(e_l,\ell_l,e_u,\ell_u)$. Then $\phi_u(T)\in\bigcup_{b = e_u+1}^{n-1} \mathbf{S}(e_l,\ell_l,b,\ell_u-1)$.
    \end{lem}
    
    \begin{proof}
        We first claim that $\phi_u(T)$ is a member of $\mathcal{T}^{(m)}(n)$. Indeed, we observe that $\lsib_{\phi_u(T)}(v_{j+2}) = \lsib_T(v_{j+1}) + \lsib_T(v_{j+2})$, $\lsib_{\phi_u(T)}(w) = \lsib_{T}(w) + \lsib_T(v_{j+1})$ is $w$ is a left sibling (but not the leftmost sibling) of $v_{j+2}$, and $\lsib_{\phi_u(T)}(w') = \lsib_T(w')$ for all other nodes $w$. Now recall from Proposition~\ref{lem:IshContribution}(2) that both $\{v_{j},\ldots,v_{j+\ell_u}\}$ and $\{v_{j+1},\ldots,v_{j+\ell_u+1}\}$ are $\mathbf{S}$-cadet sequences, but $\{v_{j},\ldots,v_{j+\ell_u+1}\}$ is not. This means
        $$\sum_{p = j+1}^{j+\ell_u+1} \lsib_T(v_p) \in S_{1,v_{j+\ell_u+1}}^-.$$
       This sum gives an upper bound on $\lsib_T(v_{j+1}) + \lsib_T(v_{j+2})$, which means $v_{j+1}$ (and hence every node) has at most $m$ left siblings in $\phi_u(T)$. We conclude that $\phi_u(T) \in \mathcal{T}^{(m)}(n)$, as desired.
       
      We now observe that $(v_1,\ldots,v_{j},v_{j+2},\ldots,v_t)$ is a maximal cadet sequence of $\phi_u(T)$. It is straightforward to verify that the valid $\mathbf{S}$-boxings of $(v_1,\ldots,v_{j},v_{j+2},\ldots,v_t)$ in $\phi_u(T)$ are precisely those of $(v_1,\ldots,v_t)$ in $T$ with the node $v_{j+1}$ removed. In particular, $(v_{j-\ell_l},\ldots,v_{j},v_{j+2},\ldots,v_{j+\ell_u})$ is the maximal $\mathbf{S}$-cadet sequence of $\phi_u(T)$ which contains $1 = v_j$. This means $L_u(\phi_u(T)) = L_u(T) - 1 = \ell_u-1$ and $L_l(\phi_u(T) = L_l(T) = \ell_l$. Moreover, if $j-\ell_l \neq 1$, the left siblings of $1 = v_j$ are the same in both $T$ and $\phi_u(T)$. This means $E_l(\phi_u(T)) = E_l(T) = e_l$.
       
     Now since  $\{v_j,v_{j+1}\}$ is an $\mathbf{S}$-cadet sequence of $T$, we see that $v_{j+\ell_u}$ is upper inefficient in $\phi_u(T)$. Moreover, if $s_i$ is a left sibling of $v_{j+1}$ in $T$ which is a node, then $s_i$ is still a left sibling of $v_{j+1}$ in $\phi_u(T)$ with $\lsib_{\phi_u(T)}(s_i) = \lsib_{T}(s_i)$. This means $s_i$ is upper inefficient in $T$ if and only if it is upper inefficient in $\phi_u(T)$. We conclude that $E_u(\phi_u(T)) > E_u(T) = e_u$.
       
       It remains to show that no other maximal cadet sequence of $\phi_u(T)$ contains a maximal $\mathbf{S}$-cadet sequence of size larger than 1. Indeed, consider the maximal cadet sequence of $\phi_u(T)$ which contains $v_{j+1}$. We note that $v_{j+1}$ is a left sibling of $v_{j+2}$, and hence is not the cadet of any node. Moreover, if $v_{j+1}$ has a cadet, then this cadet was the leftmost child of $v_{j+1}$ in $T$ and has no left siblings in $\phi_u(T)$ by construction. As no other nodes have their cadets or number of left siblings changed in moving from $T$ to $\phi_l(T)$, this implies the result.
    \end{proof}
    
Now let $T \in \cup_{b = e_u+1}^{n-1} \mathbf{S}(e_l,\ell_l,b,\ell_u-1)$. Let $(v_1,\ldots,v_j,\ldots,v_t)$ be the maximal cadet sequence of $T$ containing $1 =: v_j$. Note that $E_u(T) > 0$ implies that $j \neq t$. Now construct a new tree $\psi_u^{e_u}(T)$ as follows:
\begin{enumerate}
        \item Let $w$ be the $(e_u+1)$-th upper inefficient node of $T$ (from the right). Delete all but the leftmost child of $w$.
        \item Fix $\lambda = \lsib(v_{j+1})-\lsib(w)-1$ and denote the vertices which have $w$ as a left sibling and $v_{j+1}$ as a right sibling by $s_1,\ldots,s_{\lambda}$, indexed left to right.
        \item For each $i \in [\lambda]$ in order, delete the edge from 1 to $s_i$ and draw an edge from $w$ to $s_i$ so that $s_i$ has $i$ left siblings. Likewise, delete the edge from 1 to $v_{j+1}$ and draw an edge from $w$ to $v_{j+1}$ so that $w$ has $\lambda+1$ left siblings.
        \item Add and delete leaves as right siblings where necessary so that every node has $m+1$ children.
\end{enumerate}

As before, we emphasize that the parameter $e_l$ must be specified in the definition of $\psi_l^{e_l}$. An example is shown in Figure \ref{ex:involution3}.

\begin{lem}\label{lem:invo4}
    Let $T \in \bigcup_{b = e_u+1}^{n-1}\mathbf{S}(e_l,\ell_l,b,\ell_u-1)$. Then $\psi_u^{e_u}(T)\in\mathbf{S}(e_l,\ell_l,e_u,\ell_u)$.
\end{lem}

 \begin{proof}
        We first claim that $\psi_u^{e_u}(T)$ is a well-defined member of $\mathcal{T}^{(m)}(n)$. Indeed, we have that $\lsib_{\psi_u^{e_u}(T)}(v_{j+1}) =  \lsib_T(v_{j+1}) - \lsib_T(w)$ and $\lsib_{\phi_u(T)}(w') = \lsib_T(w')$ for all other nodes $w'$. We conclude that every node has at most $m$ left siblings in $\phi_u(T)$ and $\phi_u(T) \in \mathcal{T}^{(m)}(n)$, as desired.
       
      We now observe that $(v_1,\ldots,v_{j},w,v_{j+1},\ldots,v_t)$ is a maximal cadet sequence of $\psi_u^{e_u}(T)$. It is straightforward to verify that the valid $\mathbf{S}$-boxings of this sequence are precisely those of $(v_1,\ldots,v_t)$ in $T$ with the node $w$ appended to the maximal $\mathbf{S}$-cadet sequence which starts with $v_{j+1}$ and to the maximal $\mathbf{S}$-cadet sequence which contains $1 = v_j$. This means $L_u(\psi_u^{e_u}(T)) = L_u(T) + 1 = \ell_u$ and $L_l(\phi_u(T) = L_l(T) = \ell_l$. Moreover, if $j-\ell_l \neq 1$, the left siblings of $v_{j-\ell_l-1}$ are the same in both $T$ and $\psi_u^{e_u}(T)$. This means $E_l(\psi_u^{e_u}(T)) = E_l(T) = e_l$. 
       
       Now, since no other nodes have their cadets or number of left siblings changed in moving from $T$ to $\psi_u^{e_u}(T)$, all other maximal $\mathbf{S}$-cadet sequences of $\psi_u^{e_u}(T)$ are of size 1. Furthermore, the left siblings of $w$ in $\psi_u^{e_u}(T)$ are precisely those of $w$ in $T$ (in the same order). In particular, the upper inefficient nodes in $\psi_u^{e_u}(T)$ correspond precisely with those upper inefficient nodes of $T$ which are left of $w$, of which there are $e_u$. This implies the result.
\end{proof}

We now prove the main result of this section.

	\begin{thm}[Theorem \ref{thmB}]\label{thm:involution}
	    Let $\mathcal{A}_\mathbf{S}$ be an Ish-type arrangement. Then the regions of $\mathcal{A}_\mathbf{S}$ are equinumerous with the trees in $\mathcal{T}^{(m)}(n)$ with lower length, lower inefficiency, upper length, and upper inefficiency all 0. That is,
	    $r_{\mathbf{S}} = |\mathbf{S}(0,0,0,0)|.$
	\end{thm}
	\begin{proof}
	Let $\mathcal{T}_1^{(m)}(n) \subseteq \mathcal{T}^{(m)}(n)$ be the set of trees with nonzero contribution.
	We recall from Proposition \ref{lem:IshContribution} that
	$$r_\mathbf{S} = \sum_{T \in \mathcal{T}_1^{(m)}(n)}(-1)^{L_l(T) + L_u(T)} = \sum_{e_l+\ell_l+e_u+\ell_u<n}(-1)^{\ell_l+\ell_u}|\mathbf{S}(e_l,\ell_l,e_u,\ell_u)|.$$
	
    Now let $T \in \mathcal{T}_1^{(m)}(n)$. We note that if $L_l(T) \neq 0$, then $\phi_l(T)$ is defined. In this case, it is straightforward to show that
    $\psi_l^{E_l(T)}\circ\phi_l(T) = T$. Likewise, if $E_l(T) \neq 0$, then $\phi_l\circ\psi_l^{i}(T) = T$ for any $i \in \{0,\ldots,E_l(T)-1\}$. By replacing ``lower'' with ``upper'', the analogous statements hold for the compositions of $\phi_u$ and $\psi_u^{i}$.
    
    Now for any $e_u$ and $\ell_u$, we define a sign-reversing involution $\omega_l$ on
    $$\bigcup_{e_l,\ell_l,e_u,\ell_u:e_l + \ell_l \neq 0}\mathbf{S}(e_l,\ell_l,e_u,\ell_u)$$ by
    $$\omega_l(T) = \begin{cases}\phi_l(T) & E_l(T) = 0\\\psi_l^0(T) & E_l(T) \neq 0\end{cases}.$$
    We remark that this is indeed a sign-reversing operation since the maps $\phi_l$ and $\psi_l^0$ change the parameter $\ell_l$ by 1 and preserve the parameter $\ell_u$. Moreover, the function $\omega_l$ is an involution since for any $T$, the tree $\psi_l^0(T)$ contains no lower inefficient nodes.
    Thus we can reduce our equation to
    $$r_\mathbf{S} = \sum_{e_u + \ell_u < n} (-1)^{\ell_u}|\mathbf{S}(0,0,e_u,\ell_u)|.$$
    
    Now again, we define a sign-reversing involution $\omega_u$ on
    $$\bigcup_{e_u,\ell_u: e_u + \ell_u \neq 0}\mathbf{S}(0,0,e_u,\ell_u)$$ by
    $$\omega_u(T) = \begin{cases}\phi_u(T) & E_u(T) = 0\\\psi_u^0(T) & E_u(T) \neq 0\end{cases}.$$
	As before, this allows us to reduce equation to
    $r_\mathbf{S} = |\mathbf{S}(0,0,0,0)|.$
	\end{proof}
	
	\section{The Counting Formula for Nested Ish Arrangements}\label{sec:enumeration}
	
In this section, we show directly that for any nested Ish arrangement $\mathcal{A}_\mathbf{S}$, the number of trees in $\mathbf{S}(0,0,0,0)$ is equal to the known counting formula for number of regions of $\mathcal{A}_\mathbf{S}$ (Equation \ref{countIsh}).
	
	\begin{lem}\label{lem:countTrees}
	    Let $\mathcal{A}_\mathbf{S}$ be a nested Ish arrangement. Let $\mathfrak{T}(\mathbf{S})$ be the set of rooted labeled plane trees with $n$ nodes such that
	    \begin{enumerate}
	        \item The root is the node 1.
	        \item The node 1 has $2m+2$ children.
	        \item Every other node has one child.
	        \item For $k \neq 1$ a node, either $\lsib(k) \in S^-_{1,k}$ or $\rsib(k) \in S^-_{k,1}$.
	    \end{enumerate}
	    Then $|\mathfrak{T}(\mathbf{S})| = \prod_{k = 2}^n (n+1+|S_{1,k}| - k)$.
	\end{lem}
	
	\begin{proof}
	    Let $\mathfrak{S}(\mathbf{S})$ be the set of sequences $a_2,\ldots,a_n$ where $$a_k \in \mathfrak{A}_k:=\{(0,i)|i \in \{k+1,k+2,\ldots,n\}\} \cup \{(1,s)|s \in S_{1,k}^-\} \cup \{(-1,t)|t \in S_{k,1}^-\}$$
	    for all $k$.
	    Since $(-S_{1,k}^-)\cup S_{k,1}^- = S_{1,k}$ and $(-S_{1,k}^-) \cap S_{k,1}^- = \{0\}$, we observe that
	    $$|\mathfrak{S}(\mathbf{S})| = \prod_{k = 2}^n ((n-k)+|S_{1,k}^-| + |S_{k,1}^-|) = \prod_{k = 2}^n (n+1+|S_{1,k}| - k).$$
	We will construct a bijection between $\mathfrak{T}(\mathbf{S})$ and $\mathfrak{S}(\mathbf{S})$.
	
	For $k \in [n]$, consider the set of rooted labeled plane trees with $n-k+1$ nodes labeled $1,k+1,\ldots,n$. (Our convention when $k = n$ is that the single node is labeled 1.) Let $\mathfrak{T}_k(\mathbf{S})$ of such trees which satisfy:
	\begin{enumerate}
	    \item The root of the tree is 1.
	    \item The node 1 has $2m+2$ children.
	    \item Every other node has one child.
	    \item For $j > k$ a node, either $\lsib(j) \in S_{1,j}^-$ or $\rsib(j) \in S_{j,1}^-$.
	\end{enumerate}
	We note that $\mathfrak{T}_1(\mathbf{S}) = \mathfrak{T}(\mathbf{S})$. Now for any tree $T$ and any node $k$ in $T$ with a parent and exactly one child, we define \emph{extracting} $k$ as removing the node $k$ from the tree and drawing an edge between its neighbors.
	
	Now let $k > 1$, let $T \in \mathfrak{T}_{k-1}(\mathbf{S})$, and let $T'$ be the tree obtained by extracting the node $k$ from $T$. If $k$ is a child of 1, then we see that $\lsib_{T'}(\child(k)) = \lsib_{T}(k)$ (and hence $\rsib_{T'}(\child(k)) = \rsib_T(k)$). Moreover, for all other nodes $j > k$, we have that the number of left (and right) siblings of $j$ does not change between $T$ and $T'$. Thus since $\child(k) > k$ and $\mathcal{A}_\mathbf{S}$ is nested, $T' \in \mathfrak{T}_k(\mathbf{S})$.
	
    Now let $T \in \mathfrak{T}(\mathbf{S})$. We associate $T=:T_1$ to a sequence $a_2,\ldots,a_n$ in $\mathfrak{S}(\mathbf{S})$ as follows. For each $k \in \{2,3,\ldots,n\}$ in order:
	\begin{enumerate}
	    \item From the tree $T_{k-1}$, define
	    \begin{equation}\label{eqn:sequence}
	    a_k = \begin{cases}
	        (0,\parent(k)) & \text{if }\parent(k) \neq 1\\
	        (1,\lsib(k)) & \text{if }\parent(k) = 1\text{ and }\rsib(k) > m\\
	        (-1,\rsib(k)) & \text{if }\parent(k) = 1\text{ and }\lsib(k) > m
	    \end{cases}\end{equation}
	    \item Extract $k$ from $T_{k-1}$ to yield $T_k$.
	\end{enumerate}
		We note that the second and third cases are indeed mutually exclusive, so this map is well-defined by the previous paragraph. See Example \ref{ex:sequence} below.
	
	We next construct a map from $\mathfrak{S}(\mathbf{S})$ to $\mathfrak{T}(\mathbf{S})$. Let $a_2,a_3\ldots, a_n$ be a sequence in $\mathfrak{S}(\mathbf{S})$ and let $T_n$ be the tree consisting of a node labeled $1$ with $2m+2$ children (all of which are leaves). For each $k \in \{n,n-1, \ldots, 2\}$ in decreasing order, construct $T_{k-1}$ from $T_k$ as follows:
	\begin{enumerate}
	    \item If $a_k = (0,i)$, then $i > k$ and there is a node labeled $i$ in $T_k$. Replace the edge from $i$ to $\child(i)$ with edges from $i$ to a new node labeled $k$ and from $k$ to $\child(i)$.
	    \item If $a_k = (1,i)$, then let $c$ be the child of 1 in $T_k$ which has $i$ left siblings. Replace the edge from 1 to $c$ with edges from 1 to a new node labeled $k$ and from $k$ to $c$.
	    \item If $a_k = (-1,i)$, then let $c$ be the child of 1 in $T_k$ which has $i$ right siblings. Replace the edge from 1 to $c$ with edges from 1 to a new node labeled $k$ and from $k$ to $c$.
	\end{enumerate}
	Note that the final tree $T:=T_1$ is in $\mathfrak{T}(\mathbf{S})$.
	
	Now let $k \in [n]$ and let $T_{k} \in \mathfrak{T}_{k}(\mathbf{S})$ be arbitrary. If $k \neq 1$, let $a_k \in \mathfrak{A}_k$ be arbitrary and form a tree $T_{k-1}$ using the above construction. We then see that extracting $k$ from $T_{k-1}$ yields $T_k$ and that $a_k$ is precisely as described in Equation \ref{eqn:sequence}. Likewise, if $k \neq n$ extract $k+1$ from $T_k$ to form a new tree $T_{k+1}$ and define $a_{k+1}$ as in Equation \ref{eqn:sequence}. We then see that adding a node $k+1$ to $T_{k+1}$ using the above construction yields precisely $T_k$. We have thus given a bijection between $\mathfrak{T}(\mathbf{S})$ and $\mathfrak{S}(\mathbf{S})$.
	\end{proof}
	
\begin{ex}\label{ex:sequence}
	    Let $\mathcal{A}_\mathbf{S}$ be the nested Ish arrangement with $S_{1,j} = \{0,1,\ldots,j-2\}$ for $1 \neq j \in [6]$. Consider the tree $T \in \mathfrak{T}(\mathbf{S})$ on the right in Figure \ref{ex:lastBijection}. We see that 2 has more than $m = 4$ left siblings and 0 right siblings, meaning we associate to this tree $a_2 = (-1,0)$. Proceeding in this way, the sequence corresponding to this tree is given by:
	       $$
	        a_2 = (-1,0),\qquad
	        a_3 = (1,1),\qquad
	        a_4 = (0,6),$$
	        $$a_5 = (1,2),\qquad
	        a_6 = (-1,2).\qquad
	    $$
	\end{ex}
	
	We now recover the known counting result of \cite{AST} for the number of regions of a nested Ish arrangement.
    
	\begin{thm}[Theorem \ref{thmC}]\label{thm:Cayley}
	Let $\mathcal{A}_\mathbf{S}$ be a nested Ish arrangement. Then the number of regions of $\mathcal{A}_\mathbf{S}$ is given by
    $$r_\mathbf{S} = \prod_{k=2}^n (n+1+|S_{1,k}|-k).$$
    In particular, if $\mathcal{A}_\mathbf{S}$ is the ($n$-dimensional) Ish arrangement, then the number of regions is given by the Cayley formula: $r_\mathbf{S} = (n+1)^{n-1}$.
	\end{thm}
	\begin{proof}
	    By Theorem \ref{thm:involution} and Lemma \ref{lem:countTrees}, we need only show that there is a bijection between $\mathbf{S}(0,0,0,0)$ and $\mathfrak{T}(\mathbf{S})$. Informally, given a tree $T \in \mathbf{S}(0,0,0,0)$, we will construct a new tree as follows. If 1 is not the root of $T$, we first ``flip'' the edges pointing up from the parent of the node 1 so that they point down from 1. We then rotate everything below 1 counterclockwise until 1 is positioned as the root. See Figure~\ref{ex:lastBijection} for an example. If, on the other hand, the node 1 is already the root of $T$, we will simply add $m+1$ new leaves as the rightmost children of 1 and delete all of the children but the leftmost of every node.
	    
	    	    \begin{figure}
	    	    \centering
    	\begin{tikzpicture}
		\draw (0,0)--(0,1);
		\draw (0,1)--(-1.5,2.25);
		\draw (0,1)--(-0.75,2.25);
		\draw (0,1)--(0,2.25);
		\draw (0,1)--(0.75,2.25);
		\draw (0,1)--(1.5,2.25);
		\draw (0,2.25)--(-1.5,3.5);
		\draw (0,2.25)--(-0.75,3.5);
		\draw (0,2.25)--(0,3.5);
		\draw (0,2.25)--(0.75,3.5);
		\draw (0,2.25)--(1.5,3.5);
		\node[draw,circle,fill=white] at (0,0) {$4$};
	    \node[draw,circle,fill=white] at (0,1) {$6$};
	    \node[draw,circle,fill=white] at (0,2.25) {$1$};
	    \node[draw,circle,fill=white] at (-0.75,3.5) {$3$};
	    \node[draw,circle,fill=white] at (-1.5,2.25) {$2$};
	    \node[draw,circle,fill=white] at (0,3.5) {$5$};

         \node at (2.75,2)[anchor=center] {\Huge $\leftrightarrow$};
    \begin{scope}[shift={(7,0)}]
            \draw (0,0)--(-3,1.5);
            \draw (0,0)--(-2,1.5);
            \draw (0,0)--(-1,1.5);
            \draw (0,0)--(-0.5,1.5);
            \draw (0,0)--(-0.2,1.5);
            \draw (0,0)--(0.2,1.5);
            \draw (0,0)--(0.5,1.5);
            \draw (0,0)--(1,1.5);
            \draw (0,0)--(2,1.5);
            \draw (0,0)--(3,1.5);
            \draw (-2,1.5)--(-2,2.75);
            \draw (1,1.5)--(1,2.75);
            \draw (3,1.5)--(3,2.75);
            \draw (1,2.75)--(1,4);
            \draw (-1,1.5)--(-1,2.75);
            \node[draw,circle,fill=white] at (0,0) {$1$};
            \node[draw,circle,fill=white] at (-2,1.5) {$3$};
            \node[draw,circle,fill=white] at (1,1.5) {$6$};
            \node[draw,circle,fill=white] at (3,1.5) {$2$};
            \node[draw,circle,fill=white] at (1,2.75) {$4$};
            \node[draw,circle,fill=white] at (-1,1.5) {$5$};
        \end{scope}
	\end{tikzpicture}
	        \caption{An example of the bijection in Theorem \ref{thm:Cayley}, where $\mathcal{A}_\mathbf{S}$ is the nested Ish arrangement with $S_{1,j} = \{0,1,\ldots,j-2\}$ for $1 \neq j \in [6]$. Leaves which lie to the right of 6, above 3, or above 5 are omitted from the left picture for clarity. Visually, the forward direction of this bijection can be seen as flipping the edges pointing up from 6 (or more generally from the parent of 1) so that they point down from 1, then rotating everything below 1 counterclockwise until 1 is positioned as the root.}\label{ex:lastBijection}
	    \end{figure}
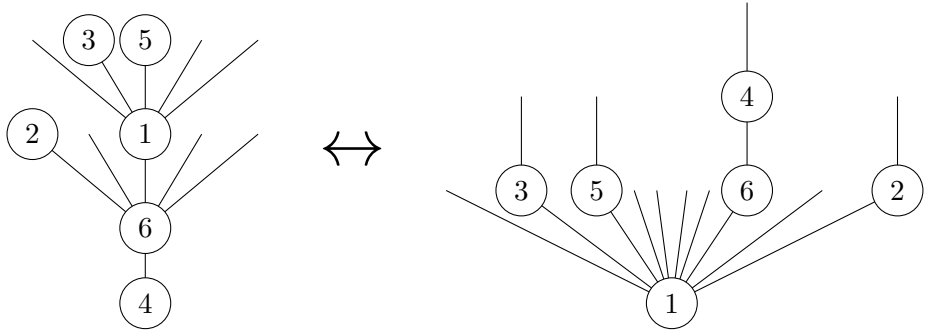
	    
	    We now rigorously define this bijection and its inverse. First let $T \in \mathbf{S}(0,0,0,0)$ and let $v$ be a node in $T$. We observe that if $v$ is not the root and is the cadet of its parent $u$, then Proposition \ref{lem:IshContribution} and the definitions of lower and upper 1-length imply that $\lsib_{T}(v) \in S_{u,v}^-$. This means that $\lsib_T(v)$ can only be nonzero if $v = 1$, $v$ is a child of 1, or $v$ is a left sibling of 1. Otherwise, either $v$ is the root or it is the leftmost child of its parent and its $m$ right siblings are all leaves.
	    
	    We now define a new tree $f(T)$ as follows. If 1 is the root of $T$, then
add $m+1$ leaves as the rightmost children of 1 and delete all of the children but the leftmost of every other node. (Note that these deleted children are all leaves.) Otherwise:
	    \begin{enumerate}
	        \item Delete all leaves other than those which are left siblings of some node, children of 1, or right siblings of 1.
	        \item Denote the left siblings of 1 by $s_1,s_2,\ldots,s_{\lsib(1)}$ and the right siblings of 1 (necessarily leaves) by $s_{\lsib(1)+2},\ldots,s_{m+1}$. 
	        \item Denote by $(v_1,\ldots,v_k,1)$ the cadet sequence from the root $v_1$ to 1. For each $j \in [k]$, delete all edges incident to $v_j$. Consider 1 as the new root.
	        \item Draw edges from 1 to $s_1,\ldots,s_{\lsib(1)},v_k,s_{\lsib(1)+2},\ldots,s_{m+1}$ so that $s_i$ has $i-1$ right siblings for each $i$ and $v_k$ has $\lsib(1)$ right siblings.
	        \item For each $j \in \{k,k-1,\ldots,2\}$ in order, draw an edge from $v_j$ to $v_{j-1}$.
	        \item Add a leaf to every node different from 1 that does not have a cadet in the new tree.
	    \end{enumerate}
	    By the preceding paragraph, we see that the tree $f(T)$ is indeed an element of $\mathfrak{T}(\mathbf{S})$.
	    
	    Now let $T \in \mathfrak{T}(\mathbf{S})$. We define a new tree $g(T)$ as follows. If there is no node amongst the $m+1$ rightmost children of $1$, delete the $m+1$ rightmost children of 1 and add leaves as right siblings of other nodes (and of any leaf which is the only child of its parent) so that every node has $m+1$ children. Otherwise,
	    \begin{enumerate}
	        \item Delete those children of nodes different from 1 that are leaves.
	        \item Denote by $s_1,\ldots,s_{m+1}$ the $m+1$ rightmost children of 1, indexed right to left. Let $s_k$ be the leftmost of these children which is a node.
	        \item Denote by $(s_k,v_1,\ldots,v_j)$ be the longest possible cadet sequence beginning with $s_k$. For $i \in [j]$, delete all edges incident to $v_i$. Consider $v_j$ as the new root.
	        \item For each $i \in \{j,j-1,\ldots,2\}$ in order, draw an edge from $v_i$ to $v_{i-1}$. Drawn an edge from $v_1$ to $s_k$.
	        \item Draw edges from $s_k$ to $s_1,\ldots,s_{k-1},1,s_{k+1},\ldots,s_{m+1}$ so that $s_i$ has $i-1$ left siblings for each $i$.
	        \item Add leaves as right siblings of nodes (and of any leaf which is the only child of its parent) as necessary so that every node has $m+1$ children.
	    \end{enumerate}
	We claim $g(T)$ is an element of $\mathbf{S}(0,0,0,0)$. We will prove this by showing every maximal $\mathbf{S}$-cadet sequence of $g(T)$ has size 1 and that $g(T)$ has no lower or upper inefficient nodes.
	
	Let $v$ be a node in $g(T)$ which has a cadet $u$. If $v = 1$, we see that $u$ is the rightmost node among the $m+1$ leftmost children of 1 in $T$ and that the left siblings of $u$ are the same in both $T$ and $g(T)$. This means $\lsib_{g(T)}(u) \in S_{1,u}^-$, so $\{1,u\}$ is not an $\mathbf{S}$-cadet sequence. Likewise, if $u = 1$, we see that $v$ is the leftmost node among the $m+1$ rightmost children of 1 in $T$, and that the left siblings of $v$ in $g(T)$ are precisely the right siblings of $v$ in $T$. This again means $\lsib_{g(T)}(1) \in S_{u,1}^-$, so $\{u,1\}$ is not an $\mathbf{S}$-cadet sequence. Finally, in the case that neither $v$ nor $u$ is the node 1, we see that either $\{u,v\}$ or $\{v,u\}$ is a cadet sequence of $T$. By construction, we then have that $\lsib_{g(T)}(u) = 0 \in S_{v,u}^-$. We conclude that every maximal $\mathbf{S}$-cadet sequence of $T$ has size 1, as claimed.
	
	Now let $u$ be a left sibling of 1 in $g(T)$. Then $u$ is one of the $m+1$ rightmost children of 1 in $T$. Moreover, the left siblings of $u$ in $g(T)$ are precisely the right siblings of $u$ in $g(T)$. This means $\lsib_{g(T)}(u) \in S_{u,1}^-$, so $u$ is not lower inefficient in $g(T)$. Likewise, let $v$ be a left sibling of the cadet of 1 in $g(T)$. Then $v$ is one of the $m+1$ leftmost children of 1 in $T$. Moreover, the left siblings of $v$ are the same in both $T$ and $g(T)$. This means $\lsib_{g(T)}(v) \in S_{1,v}^-$, so $v$ is not upper inefficient in $g(T)$. This proves the claim.
	
	It is straightforward to see that $f$ and $g$ are inverse to one another. Therefore we have a bijection between $\mathbf{S}(0,0,0,0)$ and $\mathfrak{T}(\mathbf{S})$, as desired.
	\end{proof}

\subsection*{Acknowledgements}
This work was completed as part of the MIT PRIMES (PRIMES-USA) 2020 program. The authors wish to thank the organizers of MIT PRIMES for their support and for making this collaboration possible. The authors also wish to thank Olivier Bernardi for supplying this problem and for useful guidance and suggestions. In addition, the authors are thankful to a pair of anonymous referees for offering many suggestions on how to improve this manuscript. In particular, the authors are appreciative of the referees' suggestions of references and exposition to add to the introduction. A portion of this work was completed while EH was a researcher at the Norwegian University of Science and Technology (NTNU). EH thanks NTNU for their support and hospitality.

\end{document}